\documentclass[twoside]{amsart}

\numberwithin{equation}{section}
\usepackage{amsmath,amssymb,amsfonts,amsthm,latexsym}
\usepackage{graphicx,color,enumerate}
\usepackage[all]{xy}

\newtheorem{theorem}{Theorem}[section]
\newtheorem{lemma}[theorem]{Lemma}
\newtheorem{proposition}[theorem]{Proposition}
\newtheorem{corollary}[theorem]{Corollary}
\newtheorem{conjecture}[theorem]{Conjecture}

\theoremstyle{definition}
\newtheorem{definition}[theorem]{Definition}

\theoremstyle{remark}
\newtheorem{remark}[theorem]{Remark}

\numberwithin{equation}{section}

\newcommand{\C}{\mathbb{C}}
\newcommand{\Sp}{\mathbb{S}}
\newcommand{\Q}{\mathbb{Q}}
\newcommand{\R}{\mathbb{R}}
\newcommand{\T}{\mathbb{T}}

\newcommand{\Z}{\mathbb{Z}}
\newcommand{\N}{\mathbb{N}}
\newcommand{\K}{\mathbb{K}}
\newcommand{\U}{\operatorname{U}}
\newcommand{\SU}{\operatorname{SU}}

\newcommand{\lcm}{\operatorname{lcm}}

\newcommand{\Rank}{\operatorname{rank}}
\newcommand{\diag}{\operatorname{diag}}
\newcommand{\co}{\colon\thinspace}

\newcommand{\bs}{\boldsymbol}
\newcommand{\cc} [1] {\overline {{#1}}}

\newcommand{\Imaginary}{\operatorname{Im}}
\newcommand{\Hilb}{\operatorname{Hilb}}
\newcommand{\on}{\mathit{on}}
\newcommand{\off}{\mathit{off}}

\newcommand{\calC}{\mathcal{C}}

\begin{document}

\title{The Hilbert series of a linear symplectic circle quotient}

\author{Hans-Christian Herbig}
\address{Department of Mathematics, University of Vienna, Austria} \email{herbig@imf.au.dk}

\author{Christopher Seaton}
\address{Department of Mathematics and Computer Science,
Rhodes College, 2000 N. Parkway, Memphis, TN 38112}
\email{seatonc@rhodes.edu}

\thanks{The research was supported by the \emph{Centre for the Quantum Geometry of Moduli Spaces},
which is funded by the Danish National Research Foundation. In addition, HCH was supported by the
Austrian Ministry of Science and Research BMWF, Start-Prize Y377, and CS was supported by
a Rhodes College Faculty Development Grant.}

\begin{abstract}
We compute the Hilbert series of the graded algebra of regular
functions on a symplectic quotient of a unitary circle
representation.  Additionally, we elaborate explicit formulas for the lowest
coefficients of the Laurent expansion of such a Hilbert series
in terms of rational symmetric functions of the weights.
Considerable efforts are devoted to including the cases where
the weights are degenerate.  We find that these Laurent
expansions formally resemble Laurent expansions of Hilbert
series of graded rings of real invariants of finite subgroups
of $\U_n$.  Moreover, we prove that certain Laurent
coefficients are strictly positive.  Experimental observations
are presented concerning the behavior of these coefficients as
well as relations among higher coefficients, providing
empirical evidence that these relations hold in general.
\end{abstract}

\maketitle

\tableofcontents


\section{Introduction}
\label{sec:Intro}

This  work is part of a project that attempts to elucidate under
which conditions the symplectic quotient $M_0$ (for details see
Section \ref{sec:Background}) of a unitary representation $V$
of a compact Lie group $G$ is graded-regularly symplectomorphic to a
quotient $\C^n/\Gamma$ of a finite subgroup $\Gamma$ of $\U_n$.
The notion of a \emph{regular} (respectively \emph{graded regular}) \emph{symplectomorphism} has been
introduced in \cite{FarHerSea}. Roughly speaking, a graded
regular symplectomorphism is an isomorphism of the Poisson
algebras of smooth function $\calC^\infty(M_0)\to
\calC^\infty(\C^n/\Gamma)$ that comes from lifting an
isomorphism of the $\Z$-graded Poisson algebras of regular functions
$\R[M_0]\to \R[\C^n/\Gamma]$ and satisfies suitable
semi-algebraic conditions (cf. \cite[Subsection
4.2]{FarHerSea}). So if the graded Poisson algebras $\R[M_0]$ are
$\R[\C^n/\Gamma]$ are not isomorphic,  $M_0$ and $\C^n/\Gamma$
cannot be graded-regularly symplectomorphic.  Unfortunately, it is practically
impossible to determine the affine Poisson algebras  $\R[M_0]$ and $\R[\C^n/\Gamma]$
completely for most representations. Hence, to tell $M_0$ apart from $\C^n/\Gamma$, one
needs to consider quantities that are more amenable to
computations. Our basic idea is to compare the Hilbert series of the $\Z$-graded
algebras $\R[M_0]$ and $\R[\C^n/\Gamma]$.
At the time of this writing, it is unclear whether the existence of a regular symplectomorphism between $\R[M_0]$ and $\R[\C^n/\Gamma]$ implies the existence
of a graded regular symplectomorphism between the two Poisson algebras.

In this paper, we focus on the computation of the Hilbert
series associated to unitary representations of the circle
$\Sp^1$. We discover that the Laurent expansions of the Hilbert
series of the graded rings $\R[M_0]$ and $\R[\C^n/\Gamma]$
exhibit formal similarities. Namely, in Theorems
\ref{thrm:Laurent} and \ref{thrm:GFinLaurent}, we determine the
first four Laurent coefficients of the respective Hilbert
series and find that they fulfill the same constraints.
We present empirical evidence that these constraints are the first three elements
of a sequence of linear relations satisfied by the Laurent expansions
associated to symplectic quotients of an arbitrary unitary representation; see Conjecture
\ref{conj:SymplecticTypeCompact}.
Note that  for symplectic reduction over the field $\C$ with respect to a linear cotangent lifted $\C^\times$-action, there
is an  analogue of Theorem \ref{thrm:Laurent}. Similarly, there is an analogue of Theorem \ref{thrm:GFinLaurent} for finite
subgroups of symplectic group $\operatorname{Sp}_n(\C)$. The
reader is invited to spell out the details.

As a corollary of our computations of the Laurent coefficients,
we conclude that in order for $\R[M_0]$ and
$\R[\C^n/\Gamma]$ to be isomorphic, the weights of the circle
representation have to satisfy a certain Diophantine condition
(cf. Equation \eqref{Diophant}).   We would like to stress,
however, that this condition does not lead us to any new
orbifold examples. In a forthcoming paper, we will use the
results presented here, in particular Corollary
\ref{cor:LaurentPositive}, to demonstrate that a symplectic
quotient $M_0$ of dimension $>2$ of a unitary circle
representation $V$ with $V^{\Sp^1}=\{0\}$ cannot be $\Z$-graded
regularly symplectomorphic (or $\Z$-graded regularly diffeomorphic)
to a quotient $\C^n/\Gamma$ for some finite
subgroup $\Gamma\subset \U_n$.

A major challenge in the course of proving Theorem \ref{thrm:Laurent} is to incorporate the cases when the weights degenerate,
i.e. two or more weights have the same absolute value.
A crucial observation is that in the formulas for the Laurent coefficients in Lemma \ref{lem:LaurentC}, the apparent singularities
along the diagonals are actually removable. This leads to the expressions in Theorem \ref{thrm:Laurent} in terms of Schur
polynomials that make perfect sense for the degenerate case. The main purpose of the technology around Theorem
\ref{thrm:HilbSeriesGeneral} is to show that these expressions are actually the
Laurent coefficients of the Hilbert series of $\R[M_0]$.

Apart from topological conditions (see \cite[Subsection
2.2]{FarHerSea}), we do not know of any \emph{a priori} method
in order to tell apart symplectic quotients from quotients of
finite subgroups $\U_n$. In fact it seems that orbifold cases
appear as accidents, and one has to rely on quantitative,
computational means to identify them. Conversely, if one finds
a general property of quotients of finite subgroups
$\Gamma\subset \U_n$, it is expected that the same property
holds for symplectic quotients, possibly under suitable
topological assumptions. This idea leads to Conjecture
\ref{conj:SymplecticTypeCompact},
which is explained and substantiated with sample calculations
in Section \ref{sec:HigherRelat}.

The outline of this paper is as follows. We start by reviewing some background material on invariant theory and symplectic reduction
for unitary representations of compact Lie groups in Section \ref{sec:Background}. Section \ref{sec:HilbSeries} is devoted to the
proof of Theorem \ref{thrm:HilbSeriesGeneral}, an expression for Hilbert series of symplectic circle quotients.  We
present a simple algorithm for the calculation of such a Hilbert series in Section \ref{sec:Algorithm}. In Section \ref{sec:Laurent}, we use
Theorem \ref{thrm:HilbSeriesGeneral} to calculate the lowest four Laurent coefficients of the Hilbert series of $\R[M_0]$
(see Theorem \ref{thrm:Laurent}).  To compare the Laurent coefficients to the case of finite groups, we calculate the first six Laurent
coefficients of the Hilbert series of $\R[\C^n/\Gamma]$ in Section \ref{sec:GFin}.
In Section \ref{sec:experiments}, we discuss Diophantine aspects of our findings and related experimental observations.
Section \ref{sec:HigherRelat} discusses our conjecture regarding relations among higher-degree Laurent coefficients for
the symplectic quotients considered here as well as more general symplectic quotients; this conjecture is justified
with experimental observations.

\section*{Acknowledgements}

We express our gratitude to Leonid Bedratyuk for assistance in computing invariants of $\operatorname{SL}_2$.
The research of the first author was supported by the \emph{Centre for the Quantum Geometry
of Moduli Spaces}, which is funded by the Danish National Research Foundation, and by the Austrian Ministry
of Science and Research BMWF, Start-Prize Y377.  The second author would like to thank the
\emph{Centre for the Quantum Geometry of Moduli Spaces} for hospitality
during the completion of this manuscript.

\section{Background}
\label{sec:Background}


Suppose $G\to \operatorname{U}(V)$ is a unitary representation
of the compact Lie group $G$ on a finite dimensional complex vector
space $V$ with Hermitian scalar product $\langle\:,\:\rangle$; occasionaly it will be convenient to use the shorthand $G:V$ to refer to this situation.
By convention, $\langle\:,\:\rangle$ is complex antilinear in the first
argument.
Note that we can make any symplectic representation of $G$
unitary by using an invariant compatible complex structure. Let $\cc V$ be the complex conjugate vector space of $V$.
The identity map on $V$ induces a complex antilinear map ${}^{-}\co V\to
\cc V$, $v\mapsto \cc v$. The complex conjugation ${}^{-}$ extends
to a real structure on the algebra $\mathbb C[V\times \cc V]$, and
the ring of \emph{real regular functions on} $V$ is defined to be
the subring of invariants with respect to ${}^{-}$, i.e.
$\mathbb R[V]:=\mathbb C[V\times \cc V]^-$.
It is isomorphic to the $\R$-algebra of regular functions on the
real vector space $V_{\mathbb R}$ underlying $V$. The group $G$ acts on $\cc V$ by $\cc v\mapsto (g^{-1})^t\cc v$.
Letting $G$ act on $V\times \cc V$ diagonally, and observing that
this action commutes with ${}^{-}$, we obtain an action of $G$ on
$\mathbb R[V]$ by $\mathbb R$-algebra automorphisms. This action
can be seen as coming from the  obvious $\mathbb R$-linear
$G$-action on $V_\mathbb R$. By the theorem of Hilbert and Weyl, $\mathbb R[V]^G$ is a $\N$-graded Noetherian $\mathbb R$-algebra, and we can find a \emph{Hilbert
basis}, i.e. a complete system of homogeneous polynomial invariants, $\rho_1,\dots,\rho_k \in \mathbb R[V]^G$. Note that $v\mapsto \langle v,v\rangle$ is
always a quadratic invariant.

Infinitesimally, the data of our representation are encoded by the \emph{moment map} $J$. This is the regular map from $V$ to the dual space $\mathfrak g^*$ of the Lie algebra $\mathfrak g$ of $G$ whose image of the point $v\in V$ is determined by
\[J_{\xi}(v):=J(v)(\xi)=\frac{\sqrt{-1}}{2}\langle v,\xi.v\rangle \qquad \forall  \xi\in \mathfrak g.\]
Here $\xi.v=d/dt_{|t=0}\exp(-t\xi).v$ stands for the infinitesimal action of $\xi$ on $v$.
For each $\xi\in \mathfrak g$, $J_{\xi}(v)\in \R[V]$ is  homogeneous quadratic.
Recall that the K{\"a}hler form $\omega=\Imaginary \langle\:,\:\rangle$ is a non-degenerate real-valued two-form on $V_\R$. Its inverse, the Poisson tensor $\Pi$, is a real-valued two-form on $V_\R^*$. Extending it by Leibniz rule we obtain the \emph{Poisson bracket} $\{\:,\:\}$, making $\R[V]$ a Poisson algebra over $\R$. Note that the bracket is homogeneous of degree $-2$. The fundamental vector field corresponding to $\xi\in \mathfrak g$ is the Hamiltonian vector field $\{J_\xi,\:\}$. Note also that $J:V\to \mathfrak g^*$ is $G$-equivariant and that $\{J_{\xi},J_\eta\}=J_{[\xi,\eta]}$ for $\xi,\eta\in\mathfrak g$.

Throughout the paper, we will write $Z=J^{-1}(0)$ for the preimage of zero
via the moment map. Borrowing terminology from physics, we will
occasionally refer to $Z$ as the \emph{shell}. Due to the
equivariance of $J$, $Z$ is $G$-stable. The \emph{symplectic
quotient}
\[M_0=Z/G\]
is the primary object of our investigation. For the ideal of
$Z$ in $\R[V]$ we write $I_Z$, while for the ideal
$\langle J_\xi|\xi\in\mathfrak g \rangle\subset \R[V]$ we write $I_J$.
If $Z$ is coherent, the inclusion $I_J\subset I_Z$ becomes an
equality; for more information, we refer the reader to \cite{HerbigSchwarz}.
The \emph{Poisson algebra of regular functions on $M_0$}, denoted $\R[M_0]$, is $\R[V]^G/I_Z^G$,
where $I_Z^G:=I_Z \cap \R[V]^G$ is the invariant part of the
vanishing ideal.

Recall that, given an $\N$-graded locally finite dimensional vector space $X=\bigoplus_{i\ge 0} X_i$ over the field $\K$, there is a generating function
\[\Hilb_{X|\K}(x)=\sum_{i\ge 0} \dim_\K(X_i)\:x^i\quad \in \Z[\![x]\!]\]
for the dimensions of the graded components $X_i$ of $X$. We will
refer to it as the \emph{Hilbert series} of $X$.
It will be convenient  to call elements of the ring $\R[V]^G$
\emph{off-shell invariants}. The Hilbert series of the graded ring $\R[V]^G$ will be referred to as the
\emph{off-shell Hilbert series} and we write
\[\Hilb_{G:V}^{\off}(x):=\Hilb _{\R[V]^G|\R}(x).\]
By the ring of \emph{on-shell invariants}, we mean $\R[V]^G/I_J^G$, where
\[
    I_J^G:=I_J \cap \R[V]^G
\]
is the invariant part of $I_J$. Note that if $G$ is abelian, $I_J^G$ is actually
the ideal generated by the components of the moment map over $\R[V]^G$.
By the \emph{on-shell} Hilbert series, we mean
\[
    \Hilb_{G:V}^{\on}(x):= \Hilb _{\R[V]^G/I_J^G|\R}(x).
\]

We now turn our attention to the case when $G$ is a torus $\T^\ell=(\Sp^1)^\ell$.   We identify the Lie
algebra $\mathfrak g$ of $G= \T^\ell$ with $\mathbb R^n$ by writing
an arbitrary element $(t_1,\dots,t_\ell)\in G= \T^\ell$ in the form
$t_i=\exp(2\pi\sqrt{-1}\xi_i)$ for the  vector $(\xi^1, \ldots, \xi^\ell)\in
\mathfrak g=\mathbb R^n$. We identify $V$ with $\C^n$ by choosing  coordinates $z_1,\dots,z_n$. The data of the representation are encoded by the \emph{weight matrix}
$A = (a_{ij}) \in \Z^{\ell \times n}$. Setting $(\eta_1, \ldots, \eta_n):= (\xi_1,
\ldots, \xi_\ell)\cdot A\in \mathbb R^n$, the $G= \T^\ell$-action
corresponding to the weight matrix $A$ is given by the formula
\begin{equation}
\label{eq-IntroLieAlgAction}
(t_1, \ldots, t_\ell). (z_1, \ldots, z_n)
    = (\exp(2\pi \sqrt{-1} \eta_1)z_1, \ldots , \exp(2\pi \sqrt{-1} \eta_n)z_n).
\end{equation}
The components $J_i$ of the moment map
$J=(J_1,\dots,J_n)\co \C^n \to \R^\ell \cong \mathfrak g^*$ can also be
expressed in terms of the weight matrix,
\begin{align}
    J_i(\bs z, \cc{\bs z}) =\frac{1}{2} \sum\limits_{j=1}^n a_{ij} z_j \cc z_j, \quad i = 1, \ldots, \ell,
        \quad   {\bs z} = (z_1, \ldots, z_n).
\end{align}
In the case of a torus representation, we use the notation $\Hilb_{A}^{\off}(x):=\Hilb_{G:V}^{\off}(x)$ and $\Hilb_{A}^{\on}(x):=\Hilb_{G:V}^{\on}(x)$.

\begin{proposition} If $A\in \Z^{\ell\times n}$ is a weight matrix
of $\Rank(A)=\ell$, then $\Hilb_A^{\on}(x)=(1-x^2)^\ell \Hilb_A^{\off}(x)$.
\end{proposition}
\begin{proof}
This follows from the observation that the moment map cuts out
a complete intersection of quadrics of codimension $\ell$ from
$\R[V]^G$ (cf. \cite{HerbigIyengarPflaum}).
\end{proof}

In the case of nonabelian $G$, one cannot expect such a nice relationship between on-and off-shell Hilbert series.

With the exception of Section \ref{sec:HigherRelat}, we will assume for the remainder of the paper
that $G = \Sp^1$.  Here, for simplicity, we will talk of a \emph{weight vector}
and write $A=(a_1,\dots,a_n)\in \Z^n$.  We observe that as the weight matrix of the action of $\Sp^1$ on $V\times \cc V$ is given by
$(a_1, \ldots, a_n, -a_1, \ldots, -a_n)$, $\Hilb_{A}^{\off}(x)$ and $\Hilb_{A}^{\on}(x)$ do not depend on
the signs of the weights $a_i$. For $G=\Sp^1$ we have $I_J=I_Z$
if and only if not all weights have the same sign (cf. \cite{ArmsGotayJennings, HerbigIyengarPflaum}).
Note that if all weights do have the same sign, then $Z$ is the origin and the symplectic quotient is a point.  Note further that if more than one weight is positive and more than one weight is negative,
then the symplectic quotient is not a rational homology manifold; see \cite{HerbigIyengarPflaum, FarHerSea}.

\begin{definition}
We say that the weight vector $A=(a_1,a_2,\dots, a_n)\in\Z^ n$ is \emph{generic}
if $|a_i|\ne|a_j|$ for all $i\ne j$. Otherwise, $A$ is \emph{degenerate}.
\end{definition}

It will sometimes be convenient to assume without loss of generality that
our circle action is effective, i.e. that $\gcd(a_1,\dots,a_n)=1$.
Similarly, we occasionally assume that there are no nontrivial invariants,
i.e. the weights $a_i$ are nonzero. As the signs of the weights do not
affect the Hilbert series, we also occasionally assume that the weights are non-negative.

\section{The Hilbert series of a symplectic circle quotient}
\label{sec:HilbSeries}


The aim in this section is to prove the following.

\begin{theorem}
\label{thrm:HilbSeriesGeneral}
Let $A = (a_1, \ldots, a_n)\in \Z^n$ be a nonzero weight vector for a $\Sp^1$-action on $\C^n$. The on-shell
Hilbert series (as a meromorphic function in $x$) is given by
\begin{equation}
\label{eq:HilbSeriesGeneral}
    \Hilb_A^{\scriptsize\mbox{on}}(x)
    =
    \lim\limits_{C\to A}
    \sum\limits_{i=1}^n
    \sum\limits_{\zeta^{a_i}=1}
    \frac{1}
        {c_i
        \prod\limits_{\substack{j=1 \\ j \neq i}}^n
        (1 - \zeta^{a_j} x^{(c_i + c_j)/c_i}) (1 - \zeta^{-a_j} x^{(c_i - c_j)/c_i})}.
\end{equation}
Here $C=(c_1,\dots,c_n)\in \R^n$ is assumed to be such that $|c_i|\ne |c_j|$ for $i\ne j$.
\end{theorem}

Note that if $A$ is a \emph{generic} weight vector, then
the limit is unnecessary, and the formula for the Hilbert series is obtained simply by setting
$c_i = a_i$ for $i = 1,\ldots, n$; see Proposition \ref{prop:HilbSeriesGeneric}.
For an alternative expression of the Hilbert series in the general case,
see Proposition \ref{prop:HilbSeriesDegenDeriv}.

Throughout, we will assume that $x \neq 0$ for simplicity, though the extension
of our expressions for the Hilbert series to this value can be easily checked.


\subsection{The generic case}
\label{subsec:HilbSeriesGeneric}

Let $A = (a_1, \ldots, a_n)$ be a generic weight vector.  By Molien's formula
(see e.g. \cite{DerskenKemperBook}),
the off-shell Hilbert series associated to $A$ is given by the residue
\[
   \Hilb_A^{\off}(x)= \frac{1}{2\pi i}
    \int_{z \in \Sp^1} \frac{dz}{z\prod\limits_{i=1}^n (1 - xz^{a_i}) (1 - xz^{-a_i})},
    \quad\quad |x| < 1.
\]
Rewriting the integrand $F(x,z) = F_A(x,z)$ as
\[
    F(x,z)
    =
    \frac{z^{-1 + \sum_{i=1}^n a_i}}{\prod\limits_{i=1}^n (1 - xz^{a_i}) (z^{a_i} - x)},
\]
we see that for fixed nonzero $x \in \mathbb{D}^1$, the open unit disk,
the poles of $F(x,z)$ for $z \in \mathbb{D}^1$
occur when $z^{a_i} = x$.  Fix a branch of $x^z$ by choosing a branch of the logarithm in a neighborhood
of $x$ and let $R(a_i)$ denote the set of $a_i$th roots of unity.  Then these poles occur at
$z = \zeta x^{1/a_i}$ where $i=1, \ldots, n$ and $\zeta\in R(a_i)$.

To compute the residue at a specific pole $z = \zeta_0 x^{1/a_i}$ with $\zeta_0 \in R(a_i)$, we express
\[
    F(x,z)
    =
    \left(\frac{1}{z - \zeta_0x^{1/a_i}}\right)
    \frac{z^{a_i-1}}
        {(1 - xz^{a_i})
        \prod\limits_{\substack{\zeta^{a_i}=1 \\ \zeta\neq\zeta_0}}
        (z - \zeta x^{1/a_i})
        \prod\limits_{\substack{j=1 \\ j \neq i}}^n (1 - xz^{a_j}) (1 - xz^{-a_j})}
\]
to see that this function has a simple pole at $z = \zeta_0 x^{1/a_i}$.  It follows that
\[
    \operatorname{Res}_{z=\zeta_0 x^{1/a_i}} F(x,z)
    =
    \frac{1}
        {a_i (1 - x^2)
        \prod\limits_{\substack{j=1 \\ j \neq i}}^n
            (1 - \zeta_0^{a_j} x^{(a_i + a_j)/a_i})
            (1 - \zeta_0^{a_j} x^{(a_i - a_j)/a_i})},
\]
where we simplify using the identity
\[
    \prod\limits_{\substack{\zeta^{a_i}=1 \\ \zeta\neq\zeta_0}}
        (\zeta_0 x^{1/a_i} - \zeta x^{1/a_i})
    =
    \zeta_0^{a_i-1} x^{(a_i-1)/a_i} \prod\limits_{\substack{\zeta^{a_i}=1 \\ \zeta\neq 1}}
        (1 - \zeta)
    =
    \zeta_0^{a_i-1} x^{(a_i-1)/a_i} a_i.
\]

Summing over each pole of $F(x,z)$ in $\mathbb{D}^1$ yields the following.

\begin{proposition}
\label{prop:HilbSeriesGeneric}
Let $A = (a_1, \ldots, a_n)$ be a generic weight vector.
Then the on-shell Hilbert series is given by
\begin{equation}
\label{eq:HilbSeriesGenericOn}
    \Hilb_A^{\scriptsize\mbox{on}}(x)
    =
    \sum\limits_{i=1}^n
    \sum\limits_{\zeta^{a_i}=1}
    \frac{1}
    {a_i
        \prod\limits_{\substack{j=1 \\ j \neq i}}^n
        (1 - \zeta^{a_j} x^{(a_i + a_j)/a_i}) (1 - \zeta^{-a_j} x^{(a_i - a_j)/a_i})}.
\end{equation}
\end{proposition}

Note that the expression in Equation \eqref{eq:HilbSeriesGenericOn} does not depend on the signs
of the $a_i$.

\begin{remark}
\label{rem:HilbSeriesNonEffective}
Suppose the representation with weight vector $A$ is not effective so that
$\gcd(a_1, \ldots, a_n) = a > 1$.  Let $b_i = a_i/a$ for each $i$, and then $B = (b_1, \ldots, b_n)$
is the corresponding effective representation.  From the perspective of invariant theory, it is obvious
that $\Hilb_A(x)=\Hilb_B(x)$.  To elucidate this from the perspective of the above computation,
we note that
\begin{align*}
    \Hilb_A(x)
    &=\sum\limits_{i=1}^n
    \sum\limits_{\zeta\in R(a_i)}
    \frac{1}
    {a_i
        \prod\limits_{\substack{j=1 \\ j \neq i}}^n
        (1 - \zeta^{a_j} x^{(a_i + a_j)/a_i}) (1 - \zeta^{-a_j} x^{(a_i - a_j)/a_i})}
    \\&=
    \sum\limits_{i=1}^n
    \sum\limits_{\eta\in R(b_i)}
    \sum\limits_{\zeta^a=\eta}
    \frac{1}
        {ab_i\prod\limits_{\substack{j=1 \\ j\neq i}}^n
        (1 - (\zeta^a)^{b_j} x^{(b_i + b_j)/b_i})(1 - (\zeta^a)^{-b_j} x^{(b_i - b_j)/b_i})}
    \\&=
    \sum\limits_{i=1}^n
    \sum\limits_{\eta\in R(b_i)}
        \frac{1}
        {b_i\prod\limits_{\substack{j=1 \\ j\neq i}}^n
        (1 - \eta^{b_j} x^{(b_i + b_j)/b_i})(1 - \eta^{b_j} x^{(b_i - b_j)/b_i})}
    \quad=\Hilb_B(x).
\end{align*}
Hence, there is an $a$-to-$1$ correspondence between the poles of $F_A(x,z)$ for $z\in\mathbb{D}^1$ and those of
$F_B(x,z)$ for $z\in\mathbb{D}^1$, and the residue of each pole of the latter function is $a$ times the residue
of each corresponding pole of the former.
\end{remark}


\subsection{The degenerate case}
\label{subsec:HilbSeriesDegen1}

Suppose the weight vector $A$ has degeneracies.  Then we can express
$A = (a_1, \ldots, a_1, a_2, \ldots, a_2, \ldots, a_r, \ldots, a_r)$ with each $a_i$ occurring
$p_i$ times and $a_i \neq a_j$ for $i \neq j$.  The off-shell Hilbert series associated to $A$
is given by
\[
    \frac{1}{2\pi i}
    \int_{z \in \Sp^1} \frac{dz}{z\prod\limits_{i=1}^r (1 - xz^{a_i})^{p_i} (1 - xz^{-a_i})^{p_i}}
    =
    \frac{1}{2\pi i}
    \int_{z \in \Sp^1} \frac{z^{-1 + \sum_{i=1}^r p_i a_i} \; dz}{\prod\limits_{i=1}^r (1 - xz^{a_i})^{p_i} (z^{a_i} - x)^{p_i}}
\]
with $|x| < 1$.  Fixing a branch of $x^z$, we again have that for fixed nonzero $x \in \mathbb{D}^1$,
the poles of the integrand $F_A(x,z) = F(x,z)$ with
$z \in \mathbb{D}^1$ occur at $z = \zeta x^{1/a_i}$ for $i=1, \ldots, r$ and $\zeta\in R(a_i)$.
For brevity, we set
\[
    \beta_i(x,z) = (1 - xz^{a_i})^{p_i} (1 - xz^{-a_i})^{p_i}.
\]

Now, fix an $i$ and a $\zeta_0 \in R(a_i)$, and then
\[
    F(x,z)
    =
    \left(\frac{1}{(z - \zeta_0x^{1/a_i})^{p_i}}\right)
    \frac{z^{p_i a_i-1}}
        {(1 - xz^{a_i})^{p_i}
        \prod\limits_{\substack{\zeta^{a_i}=1 \\ \zeta\neq\zeta_0}}
        (z - \zeta x^{1/a_i})^{p_i}
        \prod\limits_{\substack{j=1 \\ j \neq i}}^r \beta_j(x,z)}.
\]
Then the residue at $\zeta_0 x^{1/a_i}$ is given by the $(p_i-1)$st coefficient of the Taylor series of the
expression after the factor in parentheses, which is holomorphic at $z = \zeta_0 x^{1/a_i}$; that is,
$\operatorname{Res}_{z=\zeta_0 x^{1/a_i}} F(x,z)$ is equal to
\[
    \frac{\partial^{p_i-1}}{\partial z^{p_i-1}}
    \frac{z^{p_i a_i-1}}
        {(p_i-1)! (1 - xz^{a_i})^{p_i}
        \prod\limits_{\substack{\zeta^{a_i}=1 \\ \zeta\neq\zeta_0}}
        (z - \zeta x^{1/a_i})^{p_i}
        \prod\limits_{\substack{j=1 \\ j \neq i}}^r \beta_j(x,z)}
    \Big|_{z=\zeta_0 x^{1/a_i}}.
\]
Summing the residues yields the following.

\begin{proposition}
\label{prop:HilbSeriesDegenDeriv}
Let $A = (a_1, \ldots, a_1, a_2, \ldots, a_2, \ldots, a_r, \ldots, a_r)$ with each $a_i$ occurring
$p_i$ times and $a_i \neq a_j$ for $i \neq j$.  The off-shell Hilbert series is
\begin{equation}
\label{eq:HilbSeriesDegenDerivOff}
    \Hilb_A^{\scriptsize\mbox{off}}(x)
    =
    \sum\limits_{i=1}^r
    \sum\limits_{\zeta_0^{a_i}=1}
    \frac{\partial^{p_i-1}}{\partial z^{p_i-1}}
    G_{i}(x,z, \zeta_0)
    \Big|_{z=\zeta_0 x^{1/a_i}},
\end{equation}
where for each $i$,
\[
    G_{i}(x,z,\zeta_0) =
    \frac{z^{p_i a_i-1}}
        {(p_i-1)! (1 - xz^{a_i})^{p_i}
        \prod\limits_{\substack{\zeta^{a_i}=1 \\ \zeta\neq\zeta_0}}
        (z - \zeta x^{1/a_i})^{p_i}
        \prod\limits_{\substack{j=1 \\ j \neq i}}^r \beta_j(x,z)},
\]
and the on-shell Hilbert series is
\begin{equation}
\label{eq:HilbSeriesDegenDeriv}
    \Hilb_A^{\scriptsize\mbox{on}}(x)
    =
    \sum\limits_{i=1}^r
    \sum\limits_{\zeta_0^{a_i}=1}
    \frac{\partial^{p_i-1}}{\partial z^{p_i-1}}
    (1 - x^2)G_{i}(x,z,\zeta_0)
    \Big|_{z=\zeta_0 x^{1/a_i}}.
\end{equation}
\end{proposition}


\subsection{Another approach to the degenerate case: Proof of Theorem \ref{thrm:HilbSeriesGeneral}}
\label{subsec:HilbSeriesDegen2}

Although the formulas given in Proposition \ref{prop:HilbSeriesDegenDeriv} are useful for concrete
computations, e.g. to compute the Hilbert series associated to a specific weight vector,
we will in the sequel require the more explicit expression for the Hilbert series associated to a
degenerate weight vector given in Theorem \ref{thrm:HilbSeriesGeneral}.  In particular, the direct
relationship between the Hilbert series in the generic and degenerate cases
is not apparent from Proposition \ref{prop:HilbSeriesDegenDeriv}.

To this end, we approach the computation of the residues in the degenerate case as follows.
First express the weight vector as $A = (\bs{a},\bs{b}) = (a, \ldots, a, b_1, \ldots, b_q)$
where $a$ occurs $p$ times and $a \neq b_i$ for each $i$.  To explain this notation, note that
we are interested in computing the residues at poles corresponding to the weight $a$
(i.e. poles of the form $z = \zeta x^{1/a}$ in the computation in Subsection
\ref{subsec:HilbSeriesDegen1}).  We make no additional assumptions about the $b_i$ so that,
for instance, the weight vector $(\bs{b})$ may itself have degeneracies.

We then consider
\begin{equation}
\label{eq:HilbSeriesDegen2Integral}
    \frac{1}{2\pi i}
    \int_{z \in \Sp^1} \frac{dz}{z\prod\limits_{i=1}^p (1 - x_i z^a)(1 - x_i z^{-a})
        \prod\limits_{j=1}^q (1 - y_j z^{b_j})(1 - y_j z^{-b_j})},
\end{equation}
where $|x_i| < 1$ for $i = 1, \ldots, p$, the $x_1, \ldots, x_p$ are distinct, and
$|y_j| < 1$ for $j = 1, \ldots, q$.
Let $\bs{x} = (x_1, \ldots, x_p)$ and $\bs{y} = (y_1, \ldots, y_q)$, and as above let
\[
    F(\bs{x},\bs{y},z) = F_{(\bs{a},\bs{b})}(\bs{x},\bs{y},z)
    =
    \frac{1}{z\prod\limits_{i=1}^p (1 - x_i z^a)(1 - x_i z^{-a})
        \prod\limits_{j=1}^q (1 - y_j z^{b_j})(1 - y_j z^{-b_j})}
\]
denote the integrand.  Succinctly, the method here will be to compute the residues corresponding
to this integral and then consider the limit as $\bs{x}$ approaches a diagonal element
$\Delta_p x:= (x, \ldots, x) \in \C^p$ and $\bs{y} \to \Delta_q x = (x, \ldots, x) \in \C^q$.

First, fix a branch of the logarithm near $x$.  We will assume throughout this section that
each $x_i$ and $y_j$ is located in a connected ball $U$ about $x$ contained in the domain of this
branch, and $x^z$, $x_i^z$, and $y_j^z$ will always be defined with respect to this branch.
Then the poles of $F(\bs{x},\bs{y},z)$
for $|z|<1$ occur at $z = \zeta x_i^{1/a}$ for $i = 1, \ldots, p$ and $\zeta$ an $a$th root of unity
or $z = \eta y_j^{1/b_j}$ for $j = 1, \ldots, q$ and $\eta$ a $b_j$th root of unity.  Hence,
the integral in Equation \eqref{eq:HilbSeriesDegen2Integral} is given by
\[
    \sum\limits_{\zeta^a=1} \sum\limits_{i=1}^p
        \operatorname{Res}_{z=\zeta x_i^{1/a}}
        F(\bs{x},\bs{y},z)
    +
    \sum\limits_{j=1}^q \sum\limits_{\eta^{b_j}=1}
        \operatorname{Res}_{z=\eta y_j^{1/b_j}}
        F(\bs{x},\bs{y},z).
\]
It is easy to see that for a specific $i$ and $\zeta$,
$\lim_{\bs{x}\to\Delta x} \operatorname{Res}_{z=\zeta x_i^{1/a}} F(\bs{x},\bs{y},z)$ is not defined;
see Equation \eqref{eq:HilbSeriesDegen2Residue} below.  However, we have the following.

\begin{lemma}
\label{lem:HilbSeriesDegen2Removable}
For a fixed $a$th root of unity $\zeta_0$ and fixed $y_j$ for $j=1,\ldots, q$,
\[
    \sum\limits_{i=1}^p
        \operatorname{Res}_{z=\zeta_0 x_i^{1/a}}
        F(\bs{x},\bs{y},z)
\]
as a function of $(\bs{x}, z)$ admits an analytic continuation to the set
\[
    \{ (\bs{x}, z) : |z| < 1, |x_i| < 1, (\zeta_0 x_i^{1/a})^{b_j} \neq y_j
    \;\forall i, j\}
    \subseteq \C^{p+2}.
\]
\end{lemma}
In other words, for suitable values of the $y_k$, the singularities at points where
$x_i = x_j$ for $i \neq j$ are removable.
\begin{proof}
For brevity, set
\[
    \beta_j(y,z) = (1 - y z^{b_j})(1 - y z^{-b_j})
\]
for $j = 1, \ldots, q$.  Fixing a value of $i$ and an $a$th root of unity $\zeta_0$
and using manipulations similar to those in Subsection \ref{subsec:HilbSeriesGeneric},
we express $F(\bs{x},\bs{y},z)$ as
\[
    \left(
        \frac{1}{z - \zeta_0 x_i^{1/a}}
    \right)
    \frac{z^{a-1}}{(1 - x_i z^a)
        \prod\limits_{\substack{\zeta^a=1 \\ \zeta\neq\zeta_0}}(z - \zeta x_i^{1/a})
        \prod\limits_{\substack{j=1 \\ j \neq i}}^p (1 - x_j z^a)(1 - x_j z^{-a})
        \prod\limits_{k=1}^q \beta_k (y_k,z)}.
\]
Then when $x_i \neq x_j$ for $i \neq j$, again following the computations in Subsection \ref{subsec:HilbSeriesGeneric},

\begin{align}
\label{eq:HilbSeriesDegen2Residue}
    &\operatorname{Res}_{z=\zeta_0 x_i^{1/a}} F(\bs{x},\bs{y},z)
    \\\nonumber
    &\quad\quad\quad\quad
    =
    \frac{x_i^{p-1}}{a(1 - x_i^2)
        \prod\limits_{\substack{j=1 \\ j \neq i}}^p (1 - x_j x_i )(x_i - x_j)
        \prod\limits_{k=1}^q \beta_k (y_k,\zeta_0 x_i^{1/a})}.
\end{align}
Summing over $i$ and simplifying yields
\begin{align*}
    &\sum\limits_{i=1}^p \operatorname{Res}_{z=\zeta_0 x_i^{1/a}} F(\bs{x},\bs{y},z)
    \\
    &=
    \frac{
        \sum\limits_{i=1}^p
        (-1)^{p-i} x_i^{p-1}
            \prod\limits_{\substack{j=1 \\ j\neq i}}^p (1 - x_j^2)
            \prod\limits_{\substack{1\leq j < k \leq p \\ j,k \neq i}} (1 - x_j x_k)(x_k - x_j)
            \prod\limits_{\substack{j=1 \\ j\neq i}}^p\prod\limits_{k=1}^q \beta_k(y_k,\zeta_0 x_i^{1/a})
    }{
            a
            \prod\limits_{j=1}^p (1 - x_j^2)
            \prod\limits_{1\leq j < k \leq p} (1 - x_j x_k) (x_k - x_j)
            \prod\limits_{j=1}^p\prod\limits_{k=1}^q \beta_k(y_k,\zeta_0 x_j^{1/a})
    }.
\end{align*}

As a polynomial in $x_1, \ldots, x_p$, the numerator of this expression is alternating.  To see this, define
\begin{align*}
    &\alpha_i(x_1,\ldots,x_p)
    \\&\quad\quad=
        (-1)^{p-i} x_i^{p-1}
            \prod\limits_{\substack{j=1 \\ j\neq i}}^p (1 - x_j^2)
            \prod\limits_{\substack{1\leq j < k \leq p \\ j,k \neq i}} (1 - x_j x_k)(x_k - x_j)
            \prod\limits_{\substack{j=1 \\ j\neq i}}^p\prod\limits_{k=1}^q \beta_k(y_k,\zeta_0 x_i^{1/a})
\end{align*}
so that the numerator is equal to $\sum_{i=1}^p \alpha_i(x_1,\ldots, x_p)$.
For an odd permutation $\sigma\in\mathcal{S}_p$, an elementary computation yields that for each $i$,
\[
    \alpha_i(x_{\sigma(1)}, \ldots, x_{\sigma(p)}) = - \alpha_{\sigma(i)}(x_1, \ldots, x_p),
\]
from which it follows that the numerator is alternating.
Then as every alternating polynomial is divisible by the Vandermonde determinant
$\prod_{1\leq j < k \leq p} (x_j - x_k)$, see \cite[Section I.3]{MacdonaldBook},
it follows that there is a polynomial
$S(x_1, \ldots, x_p)$ (whose coefficients are functions of $a$, $b_1, \ldots, b_q$,
$y_1, \ldots, y_q$, and $\zeta_0$)
such that the alternating numerator can be expressed as the product
$S(x_1, \ldots, x_p)\prod_{1\leq j < k \leq p} (x_k - x_j)$.  Therefore, the singularities at $x_k = x_j$
in the sum of the residues are removable, and
$\sum_{i=1}^p \operatorname{Res}_{z=\zeta_0 x_i^{1/a}} F(\bs{x},z)$
admits the continuation
\begin{equation}
\label{eq:HilbSeriesDegen2SumUsingS}
    \frac{S(x_1, \ldots, x_p)}
        {a
        \prod\limits_{j=1}^p (1 - x_j^2)
        \prod\limits_{1\leq j < k \leq p} (1 - x_j x_k)
        \prod\limits_{j=1}^p\prod\limits_{k=1}^q \beta_k(y_k,\zeta_0 x_j^{1/a})}.
\end{equation}
This function is clearly analytic on the required domain, completing the proof.
\end{proof}

Fix an $a$th root of unity $\zeta_0$.  For each $i$, using Equation \eqref{eq:HilbSeriesDegen2Residue}
and setting $x_i = x^{t_i}$ where each $t_i$ is positive real and
$\bs{t} = (t_1, \ldots, t_p)$, we can express the limit
as $(\bs{x},\bs{y})\to(\Delta_p x,\Delta_q x)$ of
$\sum_{i=1}^p \operatorname{Res}_{z=\zeta_0 x_i^{1/a}} F(\bs{x},\bs{y},z)$ as
\[
    \lim\limits_{\bs{t}\to \Delta_p 1}\quad
    \sum\limits_{i=1}^p
        \frac{1}{a(1 - x^{2t_i})
            \prod\limits_{\substack{j=1 \\ j \neq i}}^p (1 - x^{t_j+t_i})(1 - x^{t_j-t_i})
            \prod\limits_{k=1}^q \beta_k (x,\zeta_0 x^{t_i/a})}
\]
Setting $\bs{s} = (s_1, \ldots, s_p)$, we rewrite this as
\[
    \lim\limits_{\substack{\bs{t}\to \Delta_p 1 \\ \bs{s} \to \Delta_p 1}}
    \sum\limits_{i=1}^p
        \frac{1}{(a/s_i)(1 - x^{2t_i/s_i})
            \prod\limits_{\substack{j=1 \\ j \neq i}}^p (1 - x^{(t_j+t_i)/s_j})(1 - x^{(t_j-t_i)/s_j})
            \prod\limits_{k=1}^q \beta_k (x,\zeta_0 x^{t_i/a})}
\]
so that setting $s_j = t_j$ for $j=1,\ldots,p$ yields
\begin{equation}
\label{eq:HilbSeriesDegen2Tlim}
    \lim\limits_{\bs{t}\to \Delta_p 1}\quad
    \sum\limits_{i=1}^p
        \frac{1}{(a/t_i)(1 - x^2)
            \prod\limits_{\substack{j=1 \\ j \neq i}}^p (1 - x^{(t_j+t_i)/t_j})(1 - x^{(t_j-t_i)/t_j})
            \prod\limits_{k=1}^q \beta_k (x,\zeta_0 x^{t_i/a})}
\end{equation}
where we note that
\begin{align*}
    (t_j + t_i)/t_j
    &=
    (a/t_i + a/t_j)/(a/t_i),
    \\
    (t_j - t_i)/t_j
    &=
    (a/t_i - a/t_j)/(a/t_i),
    \quad\quad\mbox{and}
    \\
    \beta_k (x,\zeta_0 x^{t_i/a})
    &=
    (1 - \zeta_0^{b_k} x^{(a + t_i b_k)/a})
    (1 - \zeta_0^{-b_k} x^{(a - t_i b_k)/a})
    \\&=
    (1 - \zeta_0^{b_k} x^{(a/t_i + b_k)/(a/t_i)})
    (1 - \zeta_0^{-b_k} x^{(a/t_i - b_k)/(a/t_i)}).
\end{align*}
Hence, setting $c_j = a/t_j$ for each $1,\ldots,p$ and $\bs{c} = (c_1, \ldots, c_p)$ yields
\begin{equation}
\label{eq:HilbSeriesDegen2Clim}
    \lim\limits_{\bs{c}\to \Delta_p a}\quad
    \sum\limits_{i=1}^p
        \frac{1}{c_i(1 - x^2)
            \prod\limits_{\substack{j=1 \\ j \neq i}}^p
                (1 - x^{(c_i + c_j)/c_i})(1 - x^{(c_i - c_j)/c_i})
            \prod\limits_{k=1}^q \beta_k (x,\zeta_0 x^{1/c_i})}
\end{equation}
where
\[
    \beta_k (x,\zeta_0 x^{1/a_i})
    =
    (1 - \zeta_0^{b_k} x^{(c_i + b_k)/c_i})
    (1 - \zeta_0^{-b_k} x^{(c_i - b_k)/c_i}).
\]

With this, we have
\begin{align*}
    &\lim\limits_{(\bs{x},\bs{y})\to(\Delta_p x,\Delta_q x)}
        \frac{1}{2\pi i}\int_{z \in \Sp^1} F(\bs{x},\bs{y},z) \; dz
    \\& \quad\quad\quad\quad
    =
    \sum\limits_{\zeta^a=1}
    \lim\limits_{\bs{x}\to\Delta_p x}
    \sum\limits_{i=1}^p
        \operatorname{Res}_{z=\zeta x_i^{1/a}}
        F(\bs{x},\Delta_q x,z)
    \\& \quad\quad\quad\quad\quad\quad\quad\quad
    +
    \lim\limits_{(\bs{x},\bs{y})\to(\Delta_p x,\Delta_q x)}
    \sum\limits_{j=1}^q \sum\limits_{\eta^{b_j}=1}
        \operatorname{Res}_{z=\eta y_j^{1/b_j}}
        F(\bs{x},\bs{y},z),
\end{align*}
where the limit of the sum over $i$ exists as above.  Switching the roles of $a$ and each
value of $b_j$ to apply the above argument to each integer that appears as a weight in $A$,
the limit of the sum over $j$ exists as well, so that the limit of integrals
on the left side of the equation exists.

Finally, it remains only to show that
\begin{equation}
\label{eq:HilbSeriesDegen2Limit}
    \lim\limits_{(\bs{x},\bs{y})\to(\Delta_p x,\Delta_q x)}
        \frac{1}{2\pi i}\int_{z \in \Sp^1} F(\bs{x},\bs{y},z) \; dz
    =
    \frac{1}{2\pi i}\int_{z \in \Sp^1}
        F(\Delta_p x,\Delta_q x,z) \; dz
    =
    \Hilb_A^{\scriptsize\mbox{off}}(x).
\end{equation}
However, note that if $\epsilon = 1 - |x|$ and each $x_i$ and $y_j$ are chosen to have
modulus bounded above by $1 - \epsilon/2$, then for $|z|=1$, $|F(\bs{x},\bs{y},z)|$ is bounded
by $2^{p+q}/\epsilon^{p+q}$.  Choose sequences
$\bs{x}_m\to\Delta_p x$ and $\bs{y}_m\to\Delta_q x$ such that for each $m$, the coordinates
of $\bs{x}_m$ and $\bs{y}_m$ are all distinct and contained in the ball of radius
$1 - \epsilon/2$ about the origin as well as the connected subset $U$ in the
domains of each of the branches of $z^{1/a}$ and $z^{1/b_j}$.  Then an application
of the Dominated Convergence Theorem yields
\[
    \lim\limits_{m\to\infty}
        \frac{1}{2\pi i}\int_{z \in \Sp^1} F(\bs{x}_m,\bs{y}_m,z) \; dz
    =
    \frac{1}{2\pi i}\int_{z \in \Sp^1}
        F(\Delta_p x,\Delta_q x,z) \; dz,
\]
so that as the limit on the left side of Equation \eqref{eq:HilbSeriesDegen2Limit}
has been shown to exist, Equation \eqref{eq:HilbSeriesDegen2Limit} follows.

With this, applying Equation \eqref{eq:HilbSeriesDegen2Clim} to
each degeneracy in the weight vector $A$ completes the proof of
Theorem \ref{thrm:HilbSeriesGeneral}.  Note that for $x$ such
that $|x| > 1$, applying the above argument to $1/x$
demonstrates that Equation \eqref{eq:HilbSeriesGeneral} holds
on the domain of the rational function
$\Hilb_A^{\scriptsize\mbox{on}}(x)$.

\section{An algorithm for computing the Hilbert series}
\label{sec:Algorithm}

In this section, we present a simple algorithm to calculate the
on-shell Hilbert series of a unitary circle representation. The
algorithm plays an important role in the logic of this work as
it provides us with plenty of empirical data. For simplicity, we
restrict to the case of a generic weight vector
$A=(a_1,\dots,a_n)$. Without loss of generality, we assume that
the weights are non-negative.

First let us fix a number $a\in\N$ and introduce an operation $U_a:\Q[\![x]\!]\to\Q[\![x]\!]$ that assigns to a
formal power series $F(x)=\sum_{\i\ge 0}F_i\: x^i$  the series
\[
    (U_a F)(x):=F_{(a)}(x):=\sum_{i\ge 0}F_{ia}\:x^i\in \Q[\![x]\!].
\]

\begin{lemma}
\label{rationalrecognition}
If $F(x)$ is a rational power series, then $(U_a F)(x)=F_{(a)}(x)$ is rational as well.
\end{lemma}
\begin{proof}
Recall the well known fact (cf. \cite[Theorem 4.1.1]{StanleyVolume1}) that $F(x)$ is rational if and only if
there exists a finite collection of polynomials $P_1,P_2,\dots P_k$ and complex numbers $\mu_1,\mu_2,\dots \mu_k$
such that $F_n=\sum_{i=1}^k P_i(n)\mu_i^n$.  Moreover, if we write $F(x)=P(x)/Q(x)$, then $Q(x)$ factors as
$Q(x)=\prod_{i=1}^k(1-\mu_i x)$. It follows that $F_{(a)}(x)$ is rational, and we can write  $F_{(a)}(x)=P_a(x)/Q_a(x)$
with  $Q_a(x)=\prod_{i=1}^k(1-\mu_i^a x)$.
\end{proof}

It is easy to see that the operation $U_a$ can be understood in terms of averaging over the cyclic groups of order $a$, i.e.
\[
    (U_a F)(x)=F_{(a)}(x)=\frac{1}{a}\sum_{\zeta^a=1}F(\zeta \sqrt[a]{x}).
\]
This operation has been used before in the context of invariant theory calculations \cite{SpringerSU2}.

Let us introduce for each $i=1,\dots,n$ the function
\[
    \widetilde{\Phi_i}(x):=\frac{1}{\prod_{j\ne i}(1-x^{a_i-a_j})(1-x^{a_i+a_j})}.
\]
Note that $\widetilde{\Phi_i}(x)$ is analytic at $x=0$. Reinterpreting formula
\eqref{eq:HilbSeriesGenericOn}, we find that
\[
    \Hilb_A^\on(x)=\sum_{i=1}^n (\widetilde{\Phi_i})_{(a_i)}(x).
\]
We observe that each $(\widetilde{\Phi_i})(x)$ can be written in the form $P(x)/Q(x)$ where $P(x)$ is a monomial and $Q(x)$
is a product of factors of the form $(1-x^m)$ with $m>0$.
We use the idea of the proof of Lemma \ref{rationalrecognition} to guess the denominator of
$(\widetilde{\Phi_i})_{(a_i)}(x)$. Namely we have to  replace each factor according to the rule
\begin{equation}
\label{eq:manipulateDenom}
    (1-x^m)\mapsto (1-x^{\lcm(a_i,m)/a_i})^{\gcd(a_i,m)}.
\end{equation}
Knowing the Taylor expansion of the rational function $(\widetilde{\Phi_i})_{(a_i)}(x)$ and its denominator polynomial,
the numerator polynomial can be determined by multiplying out. To make this an algorithm, it remains only to understand how far
we have to calculate the Taylor expansion of $\widetilde{\Phi_i}(x)$. To this end, we use Kempf's bound
(see \cite[Theorem 4.3]{Kempf79}). If we write $\Hilb_A^{\off}(x)$ as a fraction $P(x)/Q(x)$,  the bound says that
$\deg(P)\le\deg(Q)$. Consequently, we need to determine the Taylor expansion of $\widetilde{\Phi_i}(x)$ merely until
degree $a_i(d + 2)$ where $d$ is the degree of the denominator of $(\widetilde{\Phi_i})(x)$ (which coincides with the
degree of the denominator of $(\widetilde{\Phi_i})_{(a_i)}(x)$).

As an example, consider the weight vector $(1, 2, 3)$.  We have
\begin{align*}
    \widetilde{\Phi_1}(x)
        &=  \frac{1}{(1 - x^{-1})(1 - x^{3})(1 - x^{-2})(1 - x^{4})}
        &&= \frac{x^3}{(1 - x)(1 - x^{2})(1 - x^{3})(1 - x^{4})},
    \\
    \widetilde{\Phi_2}(x)
        &=  \frac{1}{(1 - x)(1 - x^{3})(1 - x^{-1})(1 - x^{5})}
        &&= \frac{-x}{(1 - x)^2(1 - x^{3})(1 - x^{5})},
    \\
    \widetilde{\Phi_3}(x)
        &=  \frac{1}{(1 - x)(1 - x^{2})(1 - x^{4})(1 - x^{5})}.
\end{align*}
Clearly $(\widetilde{\Phi_1})_{(1)}(x) = \widetilde{\Phi_1}(x)$.  Using Equation \eqref{eq:manipulateDenom},
the denominator of $(\widetilde{\Phi_2})_{(2)}(x)$ is $(1 - x)^2(1 - x^3)(1 - x^5)$ and that of
$(\widetilde{\Phi_3})_{(3)}(x)$ remains $(1 - x)(1 - x^{2})(1 - x^{4})(1 - x^{5})$.  To compute the numerator of $(\widetilde{\Phi_2})_{(2)}(x)$,
we compute the Taylor series of $\widetilde{\Phi_2}(x)$ to degree $a_2 (d + 2) = 24$ (where $d = 10$ is the degree of the denominator),
apply $U_2$, and multiply by $(1 - x)^2(1 - x^3)(1 - x^5)$.
The numerator of $(\widetilde{\Phi_2})_{(2)}(x)$ is the sum of the terms of degree at most $d = 10$ in the result,
\[
    -2 z - z^2 - 2 z^3 - z^4 - 2 z^5.
\]
Applying the same process to compute $(\widetilde{\Phi_3})_{(3)}(x)$ yields
the numerator
\[
    1 + z + 4 z^2 + 5 z^3 + 5 z^4 + 5 z^5 + 4 z^6 + z^7 + z^8,
\]
and hence the Hilbert series is given by
\begin{align*}
    &(\widetilde{\Phi_1})_{(1)}(x) + (\widetilde{\Phi_2})_{(2)}(x) + (\widetilde{\Phi_3})_{(3)}(x)
    \\&\quad\quad =
    \frac{x^3}{(1 - x)(1 - x^{2})(1 - x^{3})(1 - x^{4})}
    +
    \frac{-2 z - z^2 - 2 z^3 - z^4 - 2 z^5}{(1 - x)^2(1 - x^3)(1 - x^5)}
    \\&\quad\quad\quad\quad\quad\quad +
    \frac{1 + z + 4 z^2 + 5 z^3 + 5 z^4 + 5 z^5 + 4 z^6 + z^7 + z^8}
    {(1 - x)(1 - x^{2})(1 - x^{4})(1 - x^{5})}
    \\&\quad\quad =
    \frac{1 + z^2 + 3 z^3 + 4 z^4 + 4 z^5 + 4 z^6 + 3 z^7 + z^8 + z^{10}}
        {(1 - z^2)(1 - z^3)(1 - z^4)(1 - z^5)}.
\end{align*}

We have implemented the algorithm in a \emph{Mathematica} \cite{Mathematica} notebook available at
\texttt{http:$\backslash\backslash$faculty.rhodes.edu$\backslash$seaton$\backslash$symp\_red$\backslash$}
(use the link to the file \texttt{HilbertSeriesS1.nb}).
To give the reader an impression of what can be achieved with the notebook, let us consider
the weight vector $A=(191,192,193)\in \Z^3$. The calculation of the on-shell Hilbert series takes about
$90$ minutes on a PC. The numerator polynomial is of degree $1150$, and it takes two A4 pages to write it down.
The denominator polynomial is $(1 - z^2)(1 - z^{383})(1 - z^{384})(1 - z^{385})$.
For $n \leq 5$ and weights bounded by $10$, the evaluation takes a few seconds.

\section{Laurent expansion of $\Hilb_A^{\on}(x)$}
\label{sec:Laurent}

Throughout this section,
\[
    \Hilb_A^{\on}(x)
    =
    \sum\limits_{k=0}^\infty \gamma_k (1 - x)^{k + 2 - 2n}
\]
denotes the Laurent expansion of the on-shell Hilbert series. We shall occasionally write $\gamma_k(A)$ in order to stress the dependence of $\gamma_k$ on the data $A$. Our aim in this section is to prove the following statement.

\begin{theorem}
\label{thrm:Laurent}
Let $A = (a_1, \ldots, a_n)\in \Z^n$ be a nonzero weight vector such that $\gcd(A):=\gcd(a_1,\dots,a_n)=1$. Put $\alpha_j:=|a_j|$ for $j=1,\dots,n$ and write $\bs \alpha=(\alpha_1,\dots,\alpha_n)\in \mathbb N^n$. Let $\bs \alpha_{\widehat j}\in \mathbb N^{n-1}$
be the vector obtained from $\bs \alpha$ by omitting  $\alpha_j$ and set $g_j:=\gcd(\bs \alpha_{\widehat j})$.
 Then
\begin{align*}
    \gamma_0(A)
    &=
    \frac{s_{(n-2,n-2,n-3,\ldots,1,0)}(\bs \alpha)}
        {s_{(n-1,n-2,\ldots,1,0)}(\bs \alpha)},
    \quad\quad
    \gamma_1(A) = 0,
    \quad\quad\mbox{and}
    \\
    \gamma_2(A)= \gamma_3(A)
    &=
    \frac{\gamma_0}{12}
    +
    \frac{s_{(n-3,n-3,n-3,n-4,\ldots,1,0)}(\bs \alpha)}
        {12\:s_{(n-1,n-2,\ldots,1,0)}(\bs \alpha)}S_{\bs \alpha}
    \\
    &\quad\quad +
    \sum\limits_{j=1}^n
        \frac{(g_j^2 - 1)s_{(n-3,n-3,n-4,\ldots,1,0)}(\bs\alpha_{\widehat j})}
        {12\:s_{(n-2,n-3,\ldots,1,0)}(\bs \alpha_{\widehat j})},
\end{align*}
where $S_{\bs \alpha} = \sum_{i=1}^n \alpha_i^2$, and
$s_\lambda$ denotes the Schur polynomial associated to the
partition $\lambda$ (see Appendix \ref{ap:schur}).
\end{theorem}

We will also establish the following.


\begin{corollary}
\label{cor:LaurentPositive}
Under the assumptions of Theorem \ref{thrm:Laurent} we have $\gamma_0(A) > 0$ and $\gamma_2(A) = \gamma_3(A) > 0$.
\end{corollary}

For the rest of the section, we assume without
loss of generality that $a_i > 0$ for each $i$ (hence $A=\bs\alpha$) and $\gcd(A) = 1$;
see Section \ref{sec:Background}, and recall that the last condition corresponds to the
representation being effective.

Let $C=(c_1, \ldots, c_n)$, where each $c_i$ is a positive real number and $c_i \neq c_j$ for $i \neq j$.  Define
\[
    H_i(C, \zeta) =
    \frac{1}
        {c_i
        \prod\limits_{\substack{j=1 \\ j \neq i}}^n
        (1 - \zeta^{a_j} x^{(c_i + c_j)/c_i}) (1 - \zeta^{-a_j} x^{(c_i - c_j)/c_i})}
\]
for $i = 1, \ldots, n$ and
\[
    H(C)
    =
    \sum_{i=1}^n \sum_{\zeta^{a_i}=1} H_i (C, \zeta)
\]
so that by Equation \eqref{eq:HilbSeriesGeneral},
\[
    \Hilb_A^{\on}(x)
    =
    \lim\limits_{C\to A}
    H(C).
\]
Let $\gamma_k(i, \zeta, C)$ denote the degree $k + 2 - 2n$ coefficient in the Laurent series of
$H_i(C,\zeta)$ and let $\gamma_k(C)$ denote the degree $k + 2 - 2n$ coefficient in the
Laurent series of $H(C)$.

In Subsection \ref{subsec:LaurentGeneric}, we will compute $\gamma_k(C)$ for $k = 0, 1, 2, 3$.
This yields a formula for the coefficients
of the Hilbert series in the case that $A$ is generic by evaluating at $C = A$.
In Subsection \ref{subsec:LaurentDegen}, we will demonstrate that the Laurent series
coefficients are continuous functions of $C$ and use this to derive formulas in the general case.

\begin{remark}
\label{rem:LaurentOff}
Using the relation
\[
    \Hilb_A^{\on}(x)
    =
    (1 - x^2)\Hilb_A^{\off}(x)
\]
and setting
\[
    \Hilb_A^{\off}(x)
    =
    \sum\limits_{k=0}^\infty \delta_k (1 - x)^{k + 1 - 2n},
\]
it is easy to see that $\gamma_0 = 2\delta_0$ and $\gamma_k = 2\delta_k - \delta_{k-1}$ for $k \geq 1$.
Hence, the first four coefficients of the off-shell Laurent series can be computed directly
from the first four coefficients of the on-shell Laurent series. In view of Theorem \ref{thrm:GFinLaurent}
we would like to make the following observation.
\end{remark}

\begin{corollary}
\label{cor:LaurentOffShell}
As $\delta_1=\delta_0/2=\gamma_0/4>0$, there cannot exist a $\Z$-graded regular symplectomorphism
from the  full circle quotient $V/\Sp^1$  to some quotient $\C^{n-1}/\Gamma$ of a
finite subgroup $\Gamma \subset \U_{n-1}$.
\end{corollary}


\subsection{The Laurent series coefficients of $H(C)$.}
\label{subsec:LaurentGeneric}

Here, we demonstrate the following.
\begin{lemma}
\label{lem:LaurentC}
Let $A = (a_1, \ldots, a_n)$ be an effective weight vector with each $a_i > 0$ and let
$C = (c_1, \ldots, c_n)$ where each $c_i$ is a positive real number and
$c_i \neq c_j$ for $i \neq j$.  For each $j = 1, \ldots, n$, let $g_j = \gcd\{a_k : k\neq j\}$.
Then the first four coefficients of the Laurent series of
\[
    \sum\limits_{i=1}^n
    \sum\limits_{\zeta^{a_i}=1} H_i (C, \zeta).
\]
at $x = 1$ are given by
\begin{align}
    \label{eq:LaurentGamma01}
    \gamma_0(C)
    &=
    \sum\limits_{i=1}^n \frac{c_i^{2n-3}}
        {\prod\limits_{\substack{j=1 \\ j \neq i}}^n
        (c_i^2 - c_j^2)},
    \quad\quad\quad\quad
    \gamma_1(C) = 0,
    \quad\quad\mbox{and}
    \\
    \label{eq:LaurentGamma23}
    \gamma_2(C) = \gamma_3(C)
    &=
    -\frac{1}{12}
    \sum\limits_{i=1}^n
    \frac{c_i^{2n-5}}
        {\prod\limits_{\substack{j=1 \\ j \neq i}}^n
        (c_i^2 - c_j^2)}
    \sum\limits_{\substack{j=1\\ j\neq i}}^n c_j^2
    +
    \frac{1}{12}
    \sum\limits_{j=1}^n
    \sum\limits_{\substack{i=1 \\ i \neq j}}^n
        \frac{c_i^{2n-5}(g_j^2 - 1)}
            {\prod\limits_{\substack{k=1 \\ k \neq i,j}}^n
            (c_i^2 - c_k^2)}.
\end{align}
\end{lemma}
\begin{proof}
Using the Laurent series
\begin{equation}
\label{eq:LaurentProto}
    \frac{1}{1 - x^t}
    =
    \frac{1}{t}(1 - x)^{-1}
    + \frac{t - 1}{2t}
    + \frac{t^2 - 1}{12t}(1 - x)
    + \frac{t^2 - 1}{24t}(1 - x)^2 + \cdots,
\end{equation}
the Taylor series
\[
    \frac{1}{1 - \zeta x^t}
    =
    \frac{1}{1 - \zeta}
    - \frac{t\zeta}{(\zeta - 1)^2} (1 - x)  + \cdots,
\]
and the Cauchy product formula,
it is easy to see that for each $i$ and $a_i$th root of unity $\zeta$,
$H_i(\zeta, C)$ has a pole at $x = 1$ of order $2n - 2(k + 1)$ where
$k$ is the number of $j \in \{ 1,\ldots, n\}$ such that $\zeta^{a_j}\neq 1$.
Hence, if $\zeta^{a_j} \neq 1$ for two or more values of $j$, then $H_i(\zeta, C)$ has a pole
of order strictly smaller than $2n - 6$, implying that it does not contribute to the first
four terms of the Laurent series.

As $\gcd(a_1,\ldots,a_n) = 1$, it follows that for each $a_i$th root of unity
$\zeta\neq 1$, we have $\zeta^{a_j} \neq 1$ for at least one value of $j$, and hence
$\gamma_0(i,\zeta,C) = \gamma_1(i,\zeta,C) = 0$.  So let $\zeta = 1$, and then
using Equation \eqref{eq:LaurentProto}, the first Laurent series coefficients of
\[
    H_i(1,C)
    =
    \frac{1}
        {c_i
        \prod\limits_{\substack{j=1 \\ j \neq i}}^n
        (1 - x^{(c_i + c_j)/c_i}) (1 - x^{(c_i - c_j)/c_i})}
\]
for each $i$ are given by
\begin{align}
    \nonumber
    \gamma_0(i, 1, C)
    &=
    \frac{c_i^{2n-3}}
    {\prod\limits_{\substack{j=1 \\ j \neq i}}^n
        c_i^2 - c_j^2},
    \quad\quad\quad
    \gamma_1(i, 1, C)
    =   0,
    \quad\quad\quad\mbox{and}
    \\
    \label{eq:LaurentGamma23ZetaIs1}
    \gamma_2(i, 1, C)
    &=
    \gamma_3(i, 1, C) =
    \frac{-c_i^{2n-5}}
    {a_i \prod\limits_{\substack{j=1 \\ j \neq i}}^n
        c_i^2 - c_j^2}
    \sum\limits_{\substack{j=1\\ j\neq i}}^n c_j^2.
\end{align}
Summing over $i$ yields Equation \eqref{eq:LaurentGamma01}.

Now, suppose for some fixed $j$ that $\zeta\neq 1$ is an $a_k$th root of unity for each $k \neq j$.
The set of such $\zeta$ is precisely the set of nonunit $g_j$th roots of unity where $g_j:= \gcd\{a_k : k \neq j \}$.
This set of course may be empty, and in particular is empty if $a_j$ is a degenerate weight.

In this case, $H_i(\zeta,C)$ is given for each $i \neq j$ by
\[
    \frac{1}
        {c_i
        (1 - \zeta^{a_j} x^{(c_i + c_j)/c_i})
        (1 - \zeta^{-a_j} x^{(c_i - c_j)/c_i})
        \prod\limits_{\substack{k=1 \\ k \neq i,j}}^n
        (1 - x^{(c_i + c_k)/c_i}) (1 - x^{(c_i - c_k)/c_i})},
\]
whereby one computes that
\[
    \gamma_2(i, \zeta, C)
    =
    -\frac{c_i^{2n-5}}
        {\prod\limits_{\substack{k=1 \\ k \neq i,j}}^n
            c_i^2 - c_k^2}
    \left(
        \frac{\zeta^{a_j}}{(1 - \zeta^{a_j})^2}
    \right)
\]
and
\[
    \gamma_3(i, \zeta, C)
    =
    \frac{c_i^{2n-6}(c_j - c_i)}
        {\prod\limits_{\substack{k=1 \\ k \neq i,j}}^n
            c_i^2 - c_k^2}
    \left(
        \frac{\zeta^{a_j}}{(1 - \zeta^{a_j})^3}
    \right)
    +
    \frac{c_i^{2n-6}(c_j + c_i)}
        {\prod\limits_{\substack{k=1 \\ k \neq i,j}}^n
            c_i^2 - c_k^2}
    \left(
        \frac{\zeta^{2a_j}}{(1 - \zeta^{a_j})^3}
    \right).
\]
To compute the sum over the nonunit $g_j$th roots of unity, note the following.
As $\gcd(a_1,\ldots,a_n) = 1$, it must be that $a_j$ is coprime to $g_j = \gcd\{a_k:k\neq j\}$
and hence $\zeta\mapsto\zeta^{a_j}$ is a permutation of the nonunit $g_j$th roots of unity.
Applying a Gessel's formula \cite[Corollary 3.3]{Gessel}, one obtains
\begin{align*}
    \sum\limits_{\substack{\zeta^{g_j}=1 \\ \zeta\neq 1}}
        \frac{\zeta}{(1 - \zeta)^2}
    &=
    - \frac{g_j^2 - 1}{12},
    \quad\quad\quad
    \sum\limits_{\substack{\zeta^{g_j}=1 \\ \zeta\neq 1}}
        \frac{\zeta}{(1 - \zeta)^3}
    =
    - \frac{g_j^2 - 1}{24},
    \\
    \mbox{and}\quad&\quad
    \sum\limits_{\substack{\zeta^{g_j}=1 \\ \zeta\neq 1}}
        \frac{\zeta^2}{(1 - \zeta)^3}
    =
    \frac{g_j^2 - 1}{24}.
\end{align*}
Hence
\[
    \sum\limits_{\substack{\zeta^{g_j}=1 \\ \zeta\neq 1}}
        \gamma_2(i, \zeta, C)
    =
    \sum\limits_{\substack{\zeta^{g_j}=1 \\ \zeta\neq 1}}
        \gamma_3(i, \zeta, C)
    =
    \frac{c_i^{2n-5}(g_j^2 - 1)}
        {12\prod\limits_{\substack{k=1 \\ k \neq i,j}}^n
            c_i^2 - c_k^2}.
\]
Summing over all $i$ and $j$ yields
\[
    \sum\limits_{i=1}^n
    \sum\limits_{\substack{j=1 \\ j \neq i}}
        \gamma_2(i, \zeta, C)
    =
    \sum\limits_{i=1}^n
    \sum\limits_{\substack{j=1 \\ j \neq i}}
        \gamma_3(i, \zeta, C)
    =
    \sum\limits_{i=1}^n
    \sum\limits_{\substack{j=1 \\ j \neq i}}
        \frac{c_i^{2n-5}(g_j^2 - 1)}
            {12\prod\limits_{\substack{k=1 \\ k \neq i,j}}^n
                c_i^2 - c_k^2},
\]
which may vanish if $g_j = 1$ for each $j$.  Reordering the sum
and combining this with Equation \eqref{eq:LaurentGamma23ZetaIs1} yields Equation
\eqref{eq:LaurentGamma23}, completing the proof.
\end{proof}

\begin{remark}
\label{rem:LaurentHigherCoeffs}
Computations of higher Laurent series coefficients using the above method
become more complicated and lead to sums of the form
\[
    \sum\limits_{\substack{\zeta^{a_j} = 1 \\ \zeta\neq 1}}
    \frac{\zeta^{a_j + a_k}}{a_i(1 - \zeta^{a_j})^2(1 - \zeta^{a_k})^2}.
\]
These are special cases of \emph{Fourier--Dedekind sums}, see \cite[Section 4]{BeckDiazRobins}.
In particular, the above sum corresponds to $\sigma_{a_j + a_k}(a_j, a_j, a_k, a_k; a_i)$ in the
notation of \cite{BeckDiazRobins}, and the authors are not aware of methods of computing them
in general.
\end{remark}


\subsection{Continuity of the Laurent series coefficients}
\label{subsec:LaurentDegen}

In the case that $A$ is a generic weight vector, Lemma \ref{lem:LaurentC} can be used
to compute the Laurent series coefficients by evaluating $\gamma_k(A)$ for $k = 0, 1, 2, 3$.
If $A$ is degenerate, on the other hand, then the expressions in Equations
\eqref{eq:LaurentGamma01} and \eqref{eq:LaurentGamma23} are undefined.  In this subsection,
we demonstrate that $\lim_{C\to A} \gamma_k(C)$ is defined for a degenerate weight vector
$A$ and yields the Laurent series coefficient $\gamma_k$ of $\Hilb_A^{\on}(x)$.
Specifically, we demonstrate the following.

\begin{lemma}
\label{lem:LaurentCSchur}
With notation as above,
\begin{equation}
\label{eq:LaurentSchurGamma0}
    \gamma_0(C)
    =
    \frac{s_{(n-2,n-2,n-3,\ldots,1,0)}(C)}
        {s_{(n-1,n-2,\ldots,1,0)}(C)},
\end{equation}
and
\begin{align}
\label{eq:LaurentSchurGamma23}
    \gamma_2(C) = \gamma_3(C)
    &=
    \frac{\gamma_0(C)}{12}
    +
    \frac{s_{(n-3,n-3,n-3,n-4,\ldots,1,0)}(C)}
        {12s_{(n-1,n-2,\ldots,1,0)}(C)}S_C
    \\  \nonumber
    &\quad\quad +
    \sum\limits_{j=1}^n
        \frac{(g_j^2 - 1)s_{(n-3,n-3,n-4,\ldots,1,0)}(C_j)}
        {12s_{(n-2,n-3,\ldots,1,0)}(C_j)},
\end{align}
where $S_C = \sum_{i=1}^n c_i^2$, and $C_j \in \R^{n-1}$ is given by $C$ with $c_j$ removed.
\end{lemma}
\begin{proof}
Rewrite Equation \eqref{eq:LaurentGamma01} as
\[
    \gamma_0(C)
    =
    \frac{1}{\prod\limits_{1\leq j < k \leq n} c_j^2 - c_k^2}
    \sum\limits_{i=1}^n (-1)^{i-1} c_i^{2n-3}
    \prod\limits_{\substack{1 \leq j < k \leq n \\ j,k\neq i}}
        c_j^2 - c_k^2,
\]
and letting $\rho = (n-2,n-2,n-3,\ldots,1,0)$, observe that
\[
    \sum\limits_{i=1}^n (-1)^{i-1} c_i^{2n-3}
    \prod\limits_{\substack{1 \leq j < k \leq n \\ j,k\neq i}}
        c_j^2 - c_k^2
    =
    \det(c_i^{\rho_j + n-j})_{1\leq i,j \leq n}.
\]
This can be seen by considering the expansion of the determinant along the first
row of the matrix and noting that the minors that appear are Vandermonde determinants in the
variables $c_j^2$.  Hence
\[
    \gamma_0(C)
    =
    \frac{\det(c_i^{\rho_j + n - j})_{1\leq i,j \leq n}}
        {\prod\limits_{1\leq j < k \leq n} (c_j + c_k)(c_j - c_k)}
    =
    \frac{s_\rho(C)}
        {\prod\limits_{1\leq j < k \leq n} (c_j + c_k)}.
\]
A simple computation demonstrates that
\[
    s_{(n-1,n-2,\ldots,1,0}(C)
    =
    \prod\limits_{1\leq j < k \leq n} c_j + c_k,
\]
from which Equation \eqref{eq:LaurentSchurGamma0} follows.

Similarly, recalling that $S_C = \sum_{i=1}^n c_i^2$, we use Equation
\eqref{eq:LaurentGamma23ZetaIs1} to express
\begin{align*}
    \sum\limits_{i=1}^n \gamma_2(i,1,C)
    &=
    \sum\limits_{i=1}^n \gamma_3(i,1,C)
    =
    \frac{1}{12}
    \sum\limits_{i=1}^n
    \frac{c_i^{2n-5} (c_i^2 - S_C)}
        {\prod\limits_{\substack{j=1 \\ j \neq i}}^n
        c_i^2 - c_j^2}
    \\
    &=
    \frac{\gamma_0(C)}{12}
         -
        \frac{S_C}{12\prod\limits_{1\leq j < k \leq n} c_j^2 - c_k^2}
        \sum\limits_{i=1}^n
            (-1)^{i-1} c_i^{2n-5}
            \prod\limits_{\substack{1 \leq j < k \leq n \\ j,k\neq i}}
            c_j^2 - c_k^2.
\end{align*}
Interpreting the sum as a determinant as above, this time along the second row, yields
\[
    \sum\limits_{i=1}^n \gamma_2(i,1,C)
    =
    \sum\limits_{i=1}^n \gamma_3(i,1,C)
    =
    \frac{1}{12}
    \left(
        \gamma_0(C) +
        \frac{s_{(n-3,n-3,n-3,n-4,\ldots,1,0)}(C)}
            {s_{(n-1,n-2,\ldots,1,0)}(C)}S_C
    \right).
\]
Finally, applying a computation identical to that for $\gamma_0(C)$ to the final sum in Equation \eqref{eq:LaurentGamma23}
(in this case applied to $n - 1$ variables) yields that for each $j$,
\[
    \sum\limits_{\substack{i=1 \\ i\neq j}}^n
        \frac{c_i^{2n-5}}
            {\prod\limits_{\substack{k=1 \\ k \neq i,j}}^n
            c_i^2 - c_k^2}
    =
    \frac{s_{(n-3,n-3,n-4,\ldots,1,0)}(C_j)}
        {s_{(n-2,n-3,\ldots,1,0)}(C_j)}.
\]
Combining the above
yields Equation \eqref{eq:LaurentSchurGamma23}, completing the proof.
\end{proof}

Noting that the denominators in Equations \eqref{eq:LaurentSchurGamma0} and
\eqref{eq:LaurentSchurGamma23} are always positive when each $c_i > 0$, it is
clear that for $k = 0,1,2,3$,  $\gamma_k(A)$ is defined for a degenerate weight
matrix $A$.  To see that $\lim_{C\to A} \gamma_k(C)$ is equal to the Laurent
series coefficient $\gamma_k$ of $\Hilb_A^{\on}(x)$, note the following.
For fixed $k$,
\[
    \gamma_k(C) = \frac{1}{2\pi i} \int\limits_P H(C)(x - 1)^{k+3-2n} dx
\]
where $P$ can be taken to be a positively oriented circle with center $1$
and radius $\epsilon$ for sufficiently small $\epsilon$, requiring in particular that each $x \in P$
is contained in the domain of the branch of the logarithm used to define $x^z$.
Let $R$ denote the finite set $R =\{\zeta\neq 1 : \zeta^{\pm a_i} = 1, i=1,\ldots, n \}$.
Restricting the values of $(c_1, \ldots, c_n)$ to the compact set $\prod_{i=1}^n [a_i-1/2, a_i+1/2]$,
we may choose $\epsilon$ small enough so that for each $z$ with $|z - 1| < \epsilon$ and each
$i$ and $j$, $z^{(c_i \pm c_j)/c_i} \notin R$.  Then because the singularities of $H(C)$ at $c_i = c_j$
for $i \neq j$ are removable by Lemma \ref{lem:HilbSeriesDegen2Removable},
the integrand $H(C)(x - 1)^{k+3-2n}$ is a continuous function of $x$ and the $c_i$ on the set
$P \times R$, and hence is bounded.  Choosing a sequence $C(k) \to A$ and applying the
Dominated Convergence Theorem as in Subsection \ref{subsec:HilbSeriesDegen2}, one obtains
\begin{align*}
    \lim\limits_{k\to\infty}\frac{1}{2\pi i} \int\limits_P H(C(k))(x - 1)^{k+3-2n} dx
    &=
    \frac{1}{2\pi i} \int\limits_P \lim\limits_{k\to\infty} H(C(k))(x - 1)^{k+3-2n} dx
    \\
    &=
    \frac{1}{2\pi i} \int\limits_P H(A)(x - 1)^{k+3-2n} dx.
\end{align*}
Theorem \ref{thrm:Laurent} follows.  Moreover, it is an obvious consequence of
Lemma \ref{lem:SchurYoung} that each Schur polynomial has nonnegative coefficients
so that Corollary \ref{cor:LaurentPositive} follows immediately.


\subsection{The completely degenerate case}
\label{subsec:LaurentCompletDegen}

Suppose the absolute values of the weights all coincide; we assume without loss of generality that $A=(1,\ldots, 1)$.
We refer the reader to \cite{EgilssonMinHilbBases} for a
discussion of generators of the invariants. The Hilbert series of the corresponding invariants has a
particularly nice form, as we will see below.

We count the invariants directly, cf. \cite[Remark 1]{Hochster}.  It is easy to see that a monomial in
$z_1, \ldots, z_n$ and $\cc{z_1},\ldots,\cc{z_n}$ is invariant if and only if it is the product
of a monomial $p(z_1, \ldots, z_n)$ and a monomial $q(\cc{z_1},\ldots,\cc{z_n})$ such that $p$
and $q$ have the same degree.  Using the fact that the Hilbert series of polynomials in $n$
variables is given by
\[
    \frac{1}{(1 - x)^n}
    =
    \sum\limits_{k=0}^\infty {n + k - 1 \choose k} x^k,
\]
it follows that the Hilbert series of invariant polynomials in $z_1, \ldots, z_n, \cc{z_1},\ldots,\cc{z_n}$
is given by
\[
    \Hilb_A^{\off}(x)
    =
    \sum\limits_{k=0}^\infty {n + k - 1 \choose k}^2 x^k.
\]
This is easily seen to be equal to the hypergeometric function $_2F_1 (n, n, 1, x^2)$,
see e.g. \cite{BatemanTransFunBook}.  Hence, the on-shell Hilbert series is given by
\[
    \Hilb_A^{\on}(x)
    =
    \,_2F_1 (n, n, 1, x^2)(1 - x^2).
\]
Applying Euler's transformation \cite[I. 2.1.4 (23)]{BatemanTransFunBook}, we have
\[
    _2F_1 (n, n, 1, x^2)
    =
    \,_2F_1 (1-n, 1-n, 1, x^2)(1 - x^2)^{2n-1},
\]
which along with the definition of $_2F_1$ yields
\[
    \Hilb_A^{\on}(x)
    =
    \frac{1}{(1 - x^2)^{2n-2}}
    \sum\limits_{k=0}^{n-1} {n - 1 \choose k}^2 x^{2k}.
\]
In particular, a direct computation yields that
\[
    \gamma_0(A)
    =
    \frac{1}{2^{2n}} {2n \choose n},
\]
which by a simple induction argument is seen to be equal to the
$(n-1)$st coefficient in the Taylor series of $1/\sqrt{1 - x}$.
That is, letting, for each $n\ge 1$,  $A_n\in \Z^n$
be a completely degenerate weight vector, we have
\[
    \sum\limits_{n=0}^\infty \gamma_0(A_{n+1}) x^n
    =
    \frac{1}{\sqrt{1 - x}}.
\]

\section{Laurent expansion in the case of a finite subgroup of $\U_n$}
\label{sec:GFin}
The purpose of this section is to prove the following statement about the lowest Laurent coefficients of the ring of real invariants of a finite  subgroup of $\U_n$.
\begin{theorem}
\label{thrm:GFinLaurent}
Let $\Gamma$ be a finite subgroup of $\U_n$.  Let $\{ g_1, \ldots, g_r \}$
be a set of primitive pseudoreflections of $\Gamma$ (cf. Definition \ref{def:PrimPseudref}), let $\mathcal{Q}$ denote the
set of $g\in \Gamma$ with eigenvalue $1$ of multiplicity $n - 2$, and let
$\lambda_g$ and $\mu_g$ denote the two non-unit eigenvalues of $g \in \mathcal{Q}$.
Then the Laurent expansion of
the Hilbert series of the real invariants $\R[\C^n]^\Gamma$ at $x = 1$ is given by
\[
    \Hilb_{\R[\C^n]^\Gamma|\R}(x)
    =
    \sum\limits_{k=0}^\infty \gamma_k (1 - x)^{k - 2n}.
\]
where
\begin{align}
    \label{eq:GFinGamma01}
    \gamma_0    &=
        \frac{1}{|\Gamma|},  \quad\quad\quad\quad
    \gamma_1    =  0,
    \\
    \label{eq:GFinGamma23}
    \gamma_2    &=   \gamma_3    =
        \frac{1}{12 |\Gamma|} \sum\limits_{i=1}^r |g_i|^2 - 1,
    \\
    \label{eq:GFinGamma4}
    \gamma_4    &=
        \frac{1}{|\Gamma|}
        \left(
            \sum\limits_{i=1}^r
            \frac{-|g_i|^4 + 50|g_i|^2 - 49}{720}
            +
            \sum\limits_{g\in\mathcal{Q}}
            \frac{\lambda_g\mu_g}{(1-\lambda_g)^2(1-\mu_g)^2}
        \right),\quad\mbox{and}
    \\
    \label{eq:GFinGamma5}
    \gamma_5    &=
        \frac{1}{|\Gamma|}
        \left(
            \sum\limits_{i=1}^r
            \frac{-2|g_i|^4 + 40|g_i|^2 - 38}{720}
            +
            \sum\limits_{g\in\mathcal{Q}}
            \frac{2\lambda_g\mu_g}{(1-\lambda_g)^2(1-\mu_g)^2}
        \right).
\end{align}
In particular,
\begin{equation}
\label{eq:GFinHigherRelation}
    \gamma_3 - 2\gamma_4 + \gamma_5 = 0.
\end{equation}
\end{theorem}


For $g \in \Gamma$, we let $g_V$ denote the corresponding element of $\mathrm{U}(V)$ and $g_W$ the
corresponding element acting on $W:=V\times \cc V$.
As in the case of reduced spaces by $\Sp^1$-actions, we use the notation
\[
    \Hilb_{\R[\C^n]^\Gamma|\R}(x)
    =
    \sum\limits_{k=0}^\infty \gamma_k (1 - x)^{k - 2n}
\]
for the Laurent expansion of the Hilbert series so that $\gamma_k$ denotes the degree
$k - 2n$ coefficient.
Recall that an element $h\in\mathrm{U}(V)$ is called a \emph{pseudoreflection} if
$V^{h}$ has complex codimension $1$, or equivalently if $h$ has eigenvalue $1$ with multiplicity
$n-1$.

The Hilbert series of $\R[V]^\Gamma$ can be computed using Molien's formula for finite groups
(see e.g. \cite{SturmfelsBook}),
which for $\Gamma$ finite expresses
\[
    \Hilb_{\R[V]^\Gamma|\R}(x) =
    \frac{1}{|\Gamma|} \sum\limits_{g \in \Gamma} \frac{1}{\det(\mbox{id} - g_W^{-1}x)}.
\]
Fixing $g \in \Gamma$, choose a basis for $V$ with respect to which $g_V$ is diagonal, say
$g_V = \diag(\lambda_1, \ldots, \lambda_n)$. The ordered set of complex conjugates
of the basis elements yields a basis for $\cc{V}$, and concatenating these bases yields a basis for
$W$ with respect to which $g_W = \diag(\lambda_1, \ldots, \lambda_n, \lambda_1^{-1},\ldots,\lambda_n^{-1})$.
It follows that the corresponding term in the above sum is given by
\begin{equation}
\label{eq:GFinMolienTerm}
    \frac{1}{\det(\mbox{id} - g_W^{-1}x)}
    =
    \prod\limits_{i=1}^{n} \frac{1}{(1 - \lambda_i x)(1 - \lambda_i^{-1} x)}.
\end{equation}
For $k = 0, 1, \ldots$, we let
\[
    \prod\limits_{i=1}^{n} \frac{1}{(1 - \lambda_i x)(1 - \lambda_i^{-1} x)}
    =
    \sum\limits_{k=0}^\infty \gamma_k(g) (1 - x)^{k - 2n}
\]
denote the Laurent series expansion of the term corresponding to $g$ so that
$\gamma_k(g)/|\Gamma|$ denotes the contribution to $\gamma_k$ of the term in Molien's
formulas corresponding to $g$.  That is, for each $k$,
\[
    \gamma_k = \frac{1}{|\Gamma|} \sum\limits_{g \in \Gamma} \gamma_k(g).
\]

From Equation \eqref{eq:GFinMolienTerm}, it is clear that if $g_V$ has eigenvalue $1$ with multiplicity
$m$, then $\det(\mbox{id} - g_W^{-1}x)^{-1}$ has a pole at $x = 1$ of order $2m$.
Letting $e$ denote the identity element of $\Gamma$, it follows that
$\gamma_0(g) = \gamma_1(g) = 0$ for each $g \neq e$, $\gamma_0(e) = 1$, and $\gamma_1(e) = 0$.  Hence,
we recover the well-known fact that $\gamma_0 = 1/|\Gamma|$ as well as $\gamma_1 = 0$, proving
Equation \eqref{eq:GFinGamma01}.

Note that each $W^{g_W}$ has even complex dimension.
Therefore, there are no $g \in \Gamma$ such that $g_W$ is a pseudoreflection.  With this observation,
Equation \eqref{eq:GFinGamma01} also follow from \cite[II. Theorem 3.23]{EncMathSciAlgGeom}.

Let $\mathcal{P} = \{ g \in \Gamma : \dim_\C(V^{g_v}) = n-1 \}$, i.e. the collection of $g$ such that $g_V$ is a pseudoreflection.
We make the following.

\begin{definition}
\label{def:PrimPseudref}
A subset $\{ g_1, \ldots, g_r \} \subseteq \mathcal{P}$ will be called \emph{a set of primitive pseudoreflections}
(for the representation of $\Gamma$ on $V$) if
\begin{itemize}
\item   for each $g \in \mathcal{P}$, $g = g_i^k$ for some $i \in \{1,\ldots,r\}$ and $k\in\Z$, and
\item   $g_i^k = g_j^\ell \neq e$ implies $i = j$ and $k \cong \ell \mod |g_i|$.
\end{itemize}
\end{definition}
It is easy to see that each $\Gamma$ admits a set of primitive pseudoreflections, which is not necessarily unique.
Given a choice $\{ g_1, \ldots, g_r\}$ of primitive pseudoreflections, we have
\[
    \mathcal{P}
    =
    \bigcup\limits_{i=1}^r
    \{ g_i^k : 1 \leq k \leq |g_i|-1 \},
\]
and for each $i$, there is a basis for $V$ with respect to which
\[
    \{ g_i^k : 1 \leq k \leq |g_i|-1 \}
    =
    \bigcup\limits_{\substack{\zeta^{|g_i|} = 1\\ \zeta\neq 1}}
    \{ \diag(\zeta, 1,\ldots,1) : \zeta^{|g_i|} = 1, \zeta \neq 1 \}.
\]

For each $g \in \mathcal{P}$ with $g_V = \diag(\lambda_g,1,\ldots,1)$ (for an appropriately chosen basis for $V$), we have
\[
    \frac{1}{\det(\mbox{id} - g_W^{-1}x)}
    =
    \frac{1}{(1 - \lambda_g x)(1 - \lambda_g^{-1} x)}(1 - x)^{2 - 2n}
\]
from which it follows that
\[
    \gamma_2(g) = \gamma_3(g) = \frac{-\lambda_g}{(1 - \lambda_g)^2},
    \quad
    \gamma_4(g) = \frac{-\lambda_g^3 + \lambda_g^2 - \lambda_g}{(1 - \lambda_g)^4},
    \quad\mbox{and}\quad
    \gamma_5(g) = \frac{-\lambda_g^3 - \lambda_g}{(1 - \lambda_g)^4}.
\]
Noting that $\gamma_2(g) = \gamma_3(g) = 0$ for $g \not\in\mathcal{P}$, and choosing a set of
primitive pseudoreflections $g_1, \ldots, g_r$, we have
\begin{align*}
    \gamma_2 = \gamma_3
    &=
    \frac{1}{|\Gamma|} \sum\limits_{g\in\mathcal{P}}
        \frac{-\lambda_g}{(1 - \lambda_g)^2}
    \\
    &=
    \frac{1}{|\Gamma|} \sum\limits_{i=1}^r \sum\limits_{\substack{\zeta^{|g_i|}=1 \\ \zeta\neq 1}}
        \frac{-\zeta}{(1 - \zeta)^2}
    \\
    &=
    \frac{1}{12 |\Gamma|} \sum\limits_{i=1}^r |g_i|^2 - 1,
\end{align*}
where the last equation follows from applying Gessel's formula \cite[Theorem 4.2]{Gessel} to compute the
sum of $\zeta/(1 - t\zeta)^2$ over all $g_i$th roots of unity $\zeta$, subtracting the term corresponding
to $\zeta = 1$, and taking the limit as $t \to 1$.  This demonstrates Equation \eqref{eq:GFinGamma23}.

Similarly, we may compute the sums of $\gamma_4(g)$ and $\gamma_5(g)$ over $g\in\mathcal{P}$ by applying
Gessel's \cite[Corollary 3.3, (3.7)]{Gessel}, which yields that for each $i$,
\begin{align*}
    \sum\limits_{\substack{\zeta^{|g_i|}=1 \\ \zeta\neq 1}}
        \frac{\zeta}{(1 - \zeta)^4}
    =
    \sum\limits_{\substack{\zeta^{|g_i|}=1 \\ \zeta\neq 1}}
        \frac{\zeta^3}{(1 - \zeta)^4}
    &=
    \frac{|g_i|^4 - 20|g_i|^2 + 19}{720},
    \quad\mbox{and}
    \\
    \sum\limits_{\substack{\zeta^{|g_i|}=1 \\ \zeta\neq 1}}
        \frac{\zeta^2}{(1 - \zeta)^4}
    &=
    \frac{|g_i|^4 + 10 |g_i|^2 - 11}{720}.
\end{align*}
It follows that
\begin{align}
\label{eq:GFinGamma4PseudRef}
    \sum\limits_{g\in\mathcal{P}}
        \gamma_4(g)
    &=
    \sum\limits_{i=1}^r \sum\limits_{\substack{\zeta^{|g_i|}=1 \\ \zeta\neq 1}}
        \frac{-\zeta^3 + \zeta^2 - \zeta}{(1 - \zeta)^4}
    \\  \nonumber
    &=
    \sum\limits_{i=1}^r
        \frac{-|g_i|^4 + 50|g_i|^2 - 49}{720},
\end{align}
and
\begin{align}
\label{eq:GFinGamma5PseudRef}
    \sum\limits_{g\in\mathcal{P}}
        \gamma_5(g)
    &=
    \sum\limits_{i=1}^r \sum\limits_{\substack{\zeta^{|g_i|}=1 \\ \zeta\neq 1}}
        \frac{-\zeta^3 - \zeta}{(1 - \zeta)^4}
    \\  \nonumber
    &=
    \sum\limits_{i=1}^r
        \frac{-2|g_i|^4 + 40|g_i|^2 - 38}{720}.
\end{align}

To complete the computations of $\gamma_4$ and $\gamma_5$, we need to consider the contributions of
those $g \in \Gamma$ such that $V^{g_V}$ has complex dimension $n - 2$.  Let
$\mathcal{Q}$ denote the set of such elements, i.e. the set of $g\in \Gamma$ such that $g_V$ has eigenvalue
$1$ with multiplicity $n-2$. Then for $g \in\mathcal{Q}$ such that $g_V = \diag(\lambda_g, \mu_g, 1, \ldots, 1)$
(for an appropriate basis), we have
\[
    \frac{1}{\det(\mbox{id} - g_W^{-1}x)}
    =
    \frac{1}{(1 - \lambda_g x)(1 - \lambda_g^{-1} x)(1 - \mu_g x)(1 - \mu_g^{-1} x)}(1 - x)^{4 - 2n}
\]
so that
\[
    \gamma_4(g) = \frac{\lambda_g\mu_g}{(1-\lambda_g)^2(1-\mu_g)^2}
    \quad\mbox{and}\quad
    \gamma_5(g) = \frac{2\lambda_g\mu_g}{(1-\lambda_g)^2(1-\mu_g)^2}.
\]
Combining this with Equations \eqref{eq:GFinGamma4PseudRef} and \eqref{eq:GFinGamma5PseudRef} yields
Equations \eqref{eq:GFinGamma4} and \eqref{eq:GFinGamma5}.
As well, a simple computation from these expressions demonstrates the relation
$\gamma_3 - 2\gamma_4 + \gamma_5 = 0$ given in Equation \eqref{eq:GFinHigherRelation}.

\section{Diophantine conditions on the weights}
\label{sec:experiments}

Next, we comment on some number theoretic questions
that show up when comparing Theorems \ref{thrm:Laurent} and
\ref{thrm:GFinLaurent}. Obviously, in order for the symplectic
quotient of the circle action associated to the weight vector
$A=(a_1,\dots,a_n)$ to be $\Z$-graded regularly symplectomorphic to a
finite quotient $\C^n/\Gamma$, it is necessary that the
following Diophantine condition holds
\begin{eqnarray}\label{Diophant}
\frac{1}{\gamma_0(A)}\in \Z.
\end{eqnarray}
In a forthcoming paper, we will show that even if condition \eqref{Diophant} holds, the  symplectic
circle quotient cannot be regularly symplectomorphic to a finite quotient (after excluding the orbifold cases
treated in \cite[Section 4.3]{FarHerSea}).
Nonetheless it appears worthwhile to note that weight vectors fulfilling condition \eqref{Diophant} are rather the
exception than the rule. For simplicity, we concentrate on the case $n=3$ and assume that the  non-zero weight vector
$A=(a_1,a_2,a_3)$ has nonnegative weights. Then
\begin{align*}
\frac{1}{\gamma_0}=\frac{(a_1+a_2)(a_1+a_3)(a_2+a_3)}{a_1a_2+a_1a_3+a_2a_3}.
\end{align*}
For positive weights, this equals
\[a_1+a_2+a_3-\frac{1}{\frac{1}{a_1}+\frac{1}{a_2}+\frac{1}{a_3}},\]
and we conclude that $\gamma_0^{-1}$ is an integer if and only if the \emph{Egyptian fraction}
$ 1/a_1+1/a_2+1/a_3$ is the reciprocal of an integer. Moreover, we see that $1/\gamma_0<a_1+a_2+a_3$.
It is not difficult to show that if $1/\gamma_0\in \Z$, the weights $a_1,a_2,a_3$ cannot be
pairwise coprime. Let us introduce the \emph{level} $a_1+a_2+a_3$ and the probability $\mathcal P(L)$
to meet a weight vector with positive weights of level $\le L$ such that $\gamma_0^{-1} \in \Z$:
\begin{align*}
    &\mathcal P(L)
    =\frac{|\{(a_1,a_2,a_3)\in\N| a_1,a_2,a_3>0; a_1+a_2+a_3\le L; \gamma_0^{-1}\in\Z\}|}
        {|\{(a_1,a_2,a_3)\in\N| a_1,a_2,a_3>0; a_1+a_2+a_3\le L\}| }
    \\
    &=|\{(a_1,a_2,a_3)\in\N|  a_1,a_2,a_3>0; a_1+a_2+a_3\le L; \gamma_0^{-1}\in\Z\}|/{L\choose 3}.
\end{align*}
The graph of the function $\mathcal P(L)$ is depicted in Figure \ref{Fig:P(L)Plot} and was obtained by a case-by-case count up to level $L=3130$.
\begin{figure}[ht]
    \centering
  \includegraphics[scale=1]{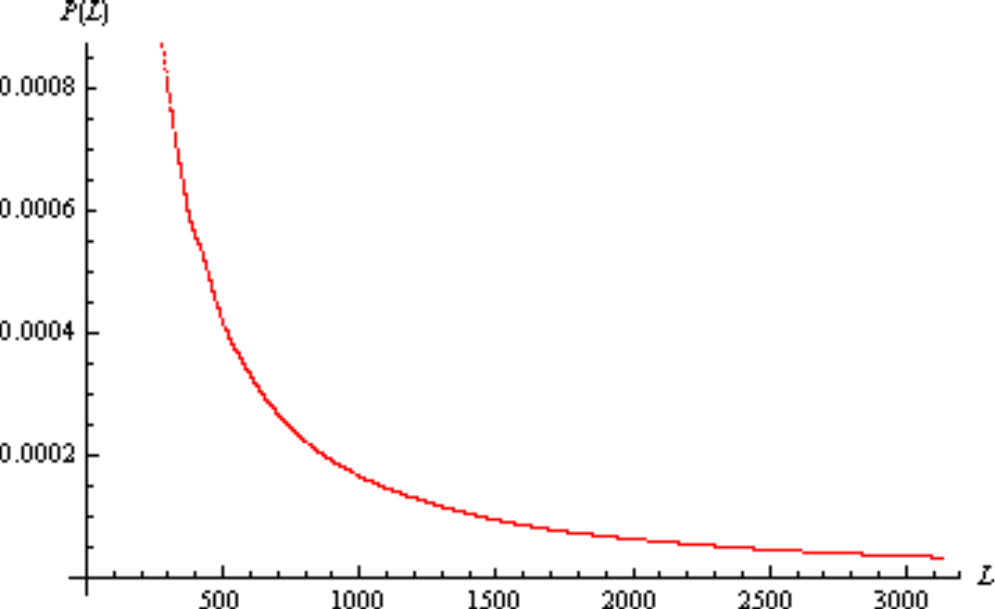}
    \caption{A plot of $\mathcal P(L)$ for small values of $L$.}
    \label{Fig:P(L)Plot}
\end{figure}
Under the assumption that, for large $L$, the probability $\mathcal P(L)$ follows approximately a power law $\mathcal P(L)\sim\alpha L^{-\beta}$,
we estimate the parameters to be $\alpha \approx 2.4105$ and $\beta \approx 1.3898$. Our experiments seem to indicate as well that
for $n\ge 4$, the probability $\mathcal P(L)$ goes to zero as $L$ goes to infinity.  We have no proof for, but believe that,
$1/\gamma_0<\sum_i a_i$ holds also for $n\ge 4$.
Due to our lack of expertise in the field, we have to leave the proof of our claims to the interested number theorist. Note
that using Equation \eqref{eq:LaurentGamma01}, it is not difficult to see that, for weight vector $A=(a,a,b)\in \Z^3$ with
$a,b>0$, $1/\gamma_0$ is never an integer. The same is true for the totally degenerate case when $n\ge 3$ (see Subsection  \ref{subsec:LaurentCompletDegen}).

Inspecting higher order Laurent coefficients (cf. Equations \eqref{eq:LaurentGamma23} and \eqref{eq:GFinGamma23}),
we note that for the symplectic circle quotient
with weight vector $A\in\Z^3$ to be regularly symplectomorphic to a $\C^2/\Gamma$ the condition
\begin{align}\label{uselessDiophant}
\frac{12\gamma_2(A)}{\gamma_0(A)}\in \Z
\end{align}
has to hold as well. This does not lead to any new insights, because our experiments suggest that
(\ref{Diophant}) implies (\ref{uselessDiophant}). However occasionally, when $1/\gamma_0$ has few
divisors, a $\Z$-graded regular symplectomorphism with an orbifold can be seen not to exist using
Equation \eqref{eq:GFinGamma23} as in the following example with $A=(4,5,20)$.
Here, $\gamma_0=1/27$ and $\gamma_2=23/162$, so that $12\gamma_2/\gamma_0=46$. It is actually
impossible to write $46$ as the sum of $(m_i^2-1)$ such that the $m_i|27$. Namely, the only
possibility is that all $m_i=3$; but $8=3^2-1$ does not divide $46$.


\section{Higher relations among the Laurent coefficients}
\label{sec:HigherRelat}

Finally, we would like to set our results in a more general context, state a conjecture,
and provide some empirical evidence.
\begin{definition}
\label{def:SymplecticType}
Let $f(x)$ be a univariate meromorphic function with pole of order $d$ at $x=1$ and consider its Laurent expansion
\[
    \sum\limits_{k=0}^\infty  \gamma_k (1 - x)^{k-d}.
\]
For each $m\ge 1$ let us introduce the linear constraint
\begin{equation}
\tag{$\operatorname S_m$}
    \sum\limits_{k=0}^{m-1} (-1)^k {m-1 \choose k} \gamma_{m+k} = 0
\end{equation}
on the coefficients $\gamma_k$. We say that $f(x)$ is \emph{symplectic of order $r$} if $(\operatorname S_m)$ holds for each $m\le r$; in this case we say for short that $f(x)$ fulfills
$(\operatorname S_{\le r})$. If $f(x)$ is symplectic of any order $r\ge 1$ we simply say that $f(x)$ is \emph{symplectic} and write $(\operatorname S_{\le \infty})$.
\end{definition}

Let us remark that odd Laurent coefficients $\gamma_3,\gamma_5,\gamma_7,\dots$ of a symplectic $f(x)$ can be reconstructed from the even Laurent
coefficients $\gamma_2,\gamma_4,\gamma_6,\dots$.

From  Theorem \ref{thrm:Laurent}, we conclude that the on-shell Hilbert series $\Hilb_{\Sp^1:V}^\on(x)$ of a unitary $\Sp^1$-representation is symplectic of order $2$.  Similarly, it follows from Theorem \ref{thrm:GFinLaurent}
that the Hilbert series $\Hilb_{\R[\C^n]^\Gamma|\R}(x)$ of real
invariants of a finite subgroup $\Gamma < \U_n$ is symplectic of order $3$.  Note that the Laurent expansion of the off-shell invariants of a
unitary circle representation do \emph{not} fulfill $(\operatorname S_1)$; see
Remark \ref{rem:LaurentOff} and Corollary \ref{cor:LaurentOffShell}.

In fact, our experiments suggest the following.

\begin{conjecture}
\label{conj:SymplecticTypeCompact}
Let $G$ be a compact Lie group, and let $G \to \U(V)$ be a unitary representation of $G$ on
a finite dimensional Hermitian vector space $V$.  Then the Hilbert series
$\Hilb_{\R[M_0]|\R}(x)$ of the graded algebra $\R[M_0]$ of regular functions on the symplectic quotient $M_0$ is symplectic.
\end{conjecture}

Note that if $G$ is finite, the symplectic quotient is simply $V/G$; hence the conjecture elaborates on the findings of
Theorem \ref{thrm:Laurent} as well as on those of Theorem \ref{thrm:GFinLaurent}. If it turns out that (possibly under some
topological assumptions) every unitary representation has this
property, then one should strive to find a conceptual (i.e. symplectic) proof that circumvents the
constructive problems of invariant theory.  For finite groups, we expect that the conjecture
might admit a simpler proof using computations similar to those in the proof of Theorem \ref{thrm:GFinLaurent}.

In the rest of the section, we describe the representations we
have tested to satisfy Conjecture \ref{conj:SymplecticTypeCompact}. For pragmatical reasons we have tested the conjecture up to order $\le 60$.  Computations were performed
using the software packages \emph{Mathematica} \cite{Mathematica},
\emph{Singular} \cite{Singular}, \emph{Normaliz} \cite{Normaliz}, and \emph{Macaulay2} \cite{M2}. In the case of binary forms of degree $d=3,4$, we used results of L. Bedratyuk based on his \emph{Maple} package \cite{BedratyukSL2}. The latter turned out to be more efficient than the algorithm of Derksen and Kemper \cite{DerskenKemperBook},
which we have implemented using \emph{Mathematica}.


\subsection{Tori}
\label{subsec:ExperimentsTorus}

For representations of $\Sp^1$ on a complex vector space of dimension $n$, we have
systematically checked $(\operatorname S_{\le 60})$ for generic weight vectors with small values of $n$ using the algorithm
outlined in Section \ref{sec:Algorithm} on \emph{Mathematica} \cite{Mathematica}.
For $n = 3$, we have checked all weight vectors with weights $\leq 15$; for $n = 4$, we have checked all weight vectors
with weights $\leq 12$; and for $n = 5$ and $n = 6$, we have checked all weight vectors
with weights $\leq 10$.  We have also checked several examples with larger weights.
The totally degenerate case has been tested for $n\le 50$ (and might be dealt with rigorously).
We also checked a few other degenerate cases such as $A=(-1,2,2)$, $A=(-2,1,1)$, $A=(-1,3,3)$ and $A=(-1,1,2,2)$
where we were able to calculate the Hilbert basis and its relations using the software packages
\emph{Singular} \cite{Singular} and \emph{Normaliz} \cite{Normaliz}.

Recall that a weight matrix is called \emph{simplicial} if $M_0$ is a rational homology manifold, which for circle
actions with $V^{\Sp^1} = \{ 0 \}$ corresponds to having one negative weight with all other weights positive (cf. \cite{FarHerSea});
note that symplectic orbifolds are necessarily rational homology manifolds.
As the signs of the weights of a circle representation do not affect the Hilbert series,
when $n \geq 4$, the on-shell Hilbert series of a circle representation may correspond to
several simplicial and non-simplicial weight vectors based on different choices of the signs of the weights.
For higher dimensional tori, there are more possibilities.
We tested a number of simplicial and non-simplicial cases. To name the simplest simplicial examples,
\[
    A=\left(\begin{array}{cccc}-1&0&1&1\\
        0&-1&0&1\end{array}\right) \mbox{ and }
    A=\left(\begin{array}{cccc}-1&0&1&1\\
        0&-1&1&1\end{array}\right),
\]
give the same Hilbert series
\[
    \frac{1 + 2x^2 + 2x^3 + 2x^4 + x^6}
    {(1 - x^2)^2 (1 - x^3)^2}
\]
with Laurent expansion
\[
    \frac{2}{9}\frac{1}{(1-x)^4}+\frac{11}{108}\frac{1}{(1-x)^2}+
        \frac{11}{108}\frac{1}{1-x}+\frac{49}{432}+ \frac{1}{8}(1-x)+\mathcal O\big((1-x)^2\big).
\]
To mention a non-simplicial weight matrix, we have tested
\[
    A=\left(\begin{array}{ccccc}-1& 0 &-1& 1& 1\\
    0& -1 &0 &-1& 1
    \end{array}\right)
\]
with Hilbert series
\[
    \frac{1 + 3x^2 + 6x^3 + 11 x^4 + 10 x^5 + 14 x^6 + 10 x^7 + 11 x^8 + 6 x^9 + 3x^{10} + x^{12}}
    {(1 - x^2)^2 (1 - x^3)^2 (1 - x^4)^2}
\]
and Laurent expansion
\[
    \frac{19}{144}\frac{1}{(1-x)^6}+ \frac{41}{864}\frac{1}{(1-x)^4}
    + \frac{41}{864}\frac{1}{(1-x)^3}+ \frac{407}{6912}\frac{1}{(1-x)^2}+ \frac{9}{128}\frac{1}{1-x}+\mathcal O(1).
\]


\subsection{Finite groups}
\label{subsec:ExperimentsFinite}

We tested that several finite subgroups of
$\U_1$ and $\U_2$ satisfy $(\operatorname S_{\le 60})$.  This included the subgroups of $\U_1$ (which are necessarily cyclic) of orders
$\leq 20$; the Hilbert series of these groups were computed using Molien's formula and the results of
\cite{Gessel} as in Section \ref{sec:GFin}.  In $\operatorname{SU}_2$, we tested the cyclic groups of orders
$\leq 20$ and binary dihedral groups of orders $4, 8, \ldots, 80$ using the formulas for the Hilbert
series given in \cite[Section 5.2]{FarHerSea}.  For finite subgroups of $\U_2$ not contained in $\operatorname{SU}_2$,
we checked twenty-two groups selected from types I, II, III, III$^\prime$, and IV according to the classification of finite subgroups of $\U_2$
in \cite[Chapter 10]{CoxeterBook}.  In these cases, the Hilbert bases and relations were computed using \emph{Singular} \cite{Singular}.


\subsection{Nonabelian non-discrete Lie groups}
\label{subsec:ExperimentsNonabel}

The case when $G$ is nonabelian is more intricate than the torus case for two reasons.
First, there is no easy way to pass from the off-shell to the on-shell Hilbert series.
Therefore, one has to calculate a Hilbert basis for the representation and then use elimination to calculate generators for the ideal $I_J^G\subset \R[V]^G$. Even if one manages to calculate the Hilbert basis,
the elimination might be practically very difficult.
Second, for most low dimensional representations, the inclusion $I_J \subset I_Z$
is proper, and the generators for $I_Z$ are usually unknown. Note that if the representation
of $G_\C$ on $V$ is $1$-large (see \cite{HerbigSchwarz}) we have $I_J=I_Z$, from which we can conclude that $I_J^G=I_Z^G$  and hence $\Hilb_{\R[M_0]|\R}(x)=\Hilb^\on_{G:V}(x)$.
We remark that, for dimensional reasons, most $1$-large representations are not amenable to computer calculations.
We expect that for a number of non-1-large representation $I_J^G=I_Z^G$ still holds.
There exist even  non-1-large cases (e.g. $\operatorname O_n:T^*\R^n$ for $n\ge 3$, cf. \cite{ArmsGotayJennings}) such that even $I_J=I_Z$ holds; these cases cannot be $0$-modular.
To our knowledge, the problem of when exactly we have $I_J^G=I_Z^G$ has not been systematically studied.


\subsubsection{$\operatorname{O}_n$ acting on $T^*(\R^n\oplus \R^n)$ for $n=2,3$}
\label{subsubsec:ExperimentsOn}

To explain an example of a nonabelian Lie group in detail, we consider the diagonal representation of  $\operatorname{O}_n$
on $V:=\C^{2n}=\R^{4n}=T^*(\R^n\oplus\R^n)$ for $n\ge 2$. By \cite[Theorem 3.5(4)]{HerbigSchwarz} the representation of $(\operatorname{O}_n)_\C$ on $V$ is $1$-large
 for $n=2,3$ and non-$1$-large for $n\ge 4$.
It has been observed by Huebschmann (see e.g. \cite{Hueb}) that for $n=3$,
this example appears as a local model for the stratification of the moduli space of flat
$\operatorname{SU}_2$-connection on a Riemann surface of genus $2$.
For the sake of simplicity, let us forget about the complex structure and work with real coordinates
$(\bs y_1,\bs y_2, \bs y_3,\bs y_4)\in\R^{4n}$, where $\bs y_i=(y_{i1},\dots, y_{in})\in \R^n$. The moment map $J$ of the representation is given by
\[\bs y_1\wedge \bs y_3+ \bs y_2\wedge \bs y_4:\R^{4n}\to \wedge^2 \R^n.\]
By the first fundamental theorem of invariant theory,
the Hilbert basis is given by the ten quadratic polynomials
$x_{ij}:=\bs y_i\cdot\bs y_j$ for $1\le i\le j\le 4$.
The inequalities defining $V/G$ have been worked out by Schwarz and Procesi \cite{SchwarzProcesi}.
In the case $n=3$, the second fundamental theorem tells us that there is  one relation, namely the determinant $D$ of the symmetric $4\times4$-matrix $X=(x_{ij})$. Hence $\R[V]^G$ is isomorphic to $\R[x_{ij}|\:1\le i\le j\le 4]/D$, where $x_{ij}$ have degree $2$ and $D$ has degree $8$. In the case $n=2$, the second fundamental theorem tells us that there are $10$ relations in degree $6$.

We use elimination theory in \emph{Macaulay2} \cite{M2} to compute the generators of the ideal $I_J^G\subset \R[V]^G$:
\begin{eqnarray*}
&x_{14}x_{33}-x_{13}x_{34}+x_{24}x_{34}-x_{23}x_{44}, \qquad x_{14}x_{23}+x_{24}^2-x_{12}x_{34}-x_{22}x_{44},\\
&x_{13}x_{23}+x_{23}x_{24}-x_{12}x_{33}-x_{22}x_{34}, \qquad x_{13}x_{14}+x_{14}x_{24}-x_{11}x_{34}-x_{12}x_{44},\\
&x_{13}^2-x_{24}^2-x_{11}x_{33}+x_{22}x_{44},  \qquad x_{12}x_{13}+x_{14}x_{22}-x_{11}x_{23}-x_{12}x_{24}, \\
& x_{24}^2x_{33}-2x_{23}x_{24}x_{34}+x_{22}x_{34}^2+x_{23}^2x_{44}-x_{22}x_{33}x_{44}, \\
&    x_{13}x_{24}^2+x_{24}^3+x_{14}x_{22}x_{34}-2x_{12}x_{24}x_{34}-x_{13}x_{22}x_{44}+x_{12}x_{23}x_{44}-x_{22}x_{24}x_{44}, \\
&     x_{11}x_{23}^2+2x_{12}x_{23}x_{24}+x_{22}x_{24}^2-x_{12}^2x_{33}-2x_{12}x_{22}x_{34}-x_{22}^2x_{44}, \\
 &    x_{14}^2x_{22}-2x_{12}x_{14}x_{24}+x_{11}x_{24}^2+x_{12}^2x_{44}-x_{11}x_{22}x_{44},\\
 &    x_{14}x_{22}x_{24}^2-x_{11}x_{23}x_{24}^2-2x_{12}x_{24}^3-x_{12}x_{14}x_{22}x_{34}
 +2x_{12}^2x_{24}x_{34}-x_{14}x_{22}^2x_{44}\\
 &-x_{12}^2x_{23}x_{44}+x_{11}x_{22}x_{23}x_{44}+2x_{12}x_{22}x_{24}x_{44}.
\end{eqnarray*}
We have six generators of degree $4$, four generators of degree $6$ and one generator of degree $8$.
Experimentation indicates that the result actually does not depend on $n\ge 2$, which is surprising,
because the off-shell quotient does depend on $n$.
Using \emph{Macaulay2}, we are able to compute the minimal free resolution of the ring homomorphism $\R[V]^G\to\R[V]^G/I_J\cap \R[V]^G$, which turns out to be pure.
Accordingly, the Hilbert series of the ring of on-shell invariants is
\[
    \Hilb_{\operatorname{O}_n:\R^{4n}}^\on(x)
    =\frac{1-6x^4+5x^6+5x^8-6x^{10}+x^{14}}{(1-x^2)^{10}}.
\]
The order of the pole at $x=1$ is $6$, and the first Laurent coefficients are
\[
    \gamma_0=5/32,\quad \gamma_1= 0,\quad \gamma_2=\gamma_3=\gamma_4=\gamma_5= 11/128.
\]
It is symplectic up to order $\le 60$.
We conclude that for $n=2,3$ the symplectic quotient $M_0$ is $6$-dimensional and cannot be $\Z$-graded symplectomorphic to a finite quotient. The question whether all these $M_0$ are symplectomorphic for $n\ge 2$ will be addressed elsewhere.

\subsubsection{$\operatorname{SU}_2:2V_1$}
\label{counterexample}
We include this example to illustrate that, in contrast to the Hilbert series of regular functions,
the on-shell Hilbert series is not necessarily
symplectic. We let $\operatorname{SU}_2$ act diagonally on $2V_1=V_1\oplus V_1$, where $V_1$ denotes the binary forms of degree $1$. The representation of $\operatorname{SL}_2=(\operatorname{SU}_2)_\C$ is \emph{not} $1$-large. Calculating the on-shell Hilbert series along the lines of Subsection \ref{subsubsec:ExperimentsOn}, we find
\begin{eqnarray*}
\Hilb^\on_{\operatorname{SU}_2:2V_1}(x)&=&\frac{1 - 11x^4  + 24x^6  - 21x^8  + 8x^{10}   - x^{12}}{(1 - x^2 )^6}\\
&=&\frac{1}{(1-x)^2} - \frac{3}{4} + \frac{1}{4}(1-x) + \frac{3}{16} (1-x)^2+\mathcal O((1-x)^3).
\end{eqnarray*}
This Hilbert series violates the condition $(\operatorname{S}_r)$ for $r=2$, but it does fulfill $(\operatorname{S}_r)$ for
$r\in\{1,\dots, N\}\backslash\{2\}$ when $N=100$ (and probably for arbitrary $N$). Moreover, it is not palindromic, hence the ring $\R[V]^G/I_J^G$ cannot be Gorenstein. It has been shown in \cite{ArmsGotayJennings} that the generators of $I_Z$ can be identified with the moment map of the $\operatorname{SO}_4$ action on $T^* \R^4\cong 2V_1$.
Based on this observation, it is easy to write down a $\Z$-graded regular symplectomorphism from the symplectic quotient
of the representation $\operatorname{SU}_2:2V_1$ to the finite quotient $\C/\Z_2$,
which has been checked to fulfill $(\operatorname{S}_{\le 60})$ (cf. Subsection \ref{subsec:ExperimentsFinite}).

\subsubsection{Rediscovering finite quotients}
\label{rediscover}
Let us give a (non-exhaustive) list of representations where $\Hilb_{\R[M_0]}(x)$ coincides with that of a finite quotient. For the representations $\operatorname{SO}_n:T^*\R^n$ and  $\operatorname{O}_n:T^*\R^n$ for  $n\ge 2$ as well as $\operatorname{SU}_2:2V_1$, the Hilbert series $\Hilb_{\R[M_0]}(x)$ coincides with that of the finite quotient $\C/\Z_2$.
The representations $\operatorname{SU}_2: V_3$ (respectively $\operatorname{SU}_2: V_4$) give rise to the same Hilbert series as the finite quotients $\C/\Z_4$ (respectively
$\C^2/\Gamma$ with $\Gamma$ the subgroup of $\U_2$ of type III$^\prime$ that is
denoted by $(\Z_4/1, \mathbb{D}_3/\Z_3)$ in \cite[Chapter 10]{CoxeterBook}).
These finite quotients have been checked to fulfill $(\operatorname{S}_{\le 60})$ (cf. Subsection \ref{subsec:ExperimentsFinite}).

\subsubsection{Non-orbifold symplectic quotients}
\label{nonorb}
The following $1$-large representations each give rise to a $\Hilb_{\R[M_0]}(x)$   that is symplectic of
order $\le 60$ and has a leading coefficient $\gamma_0$ such that  $1/\gamma_0$ is not an integer: $\operatorname {O}_n:T^*(2\R^n)$ for $n=2,3,4$ (cf. Subsection \ref{subsubsec:ExperimentsOn}),
$\operatorname {SO}_3:T^*(2\R^3)$, $\operatorname {O}_2:T^*(3\R^2)$, $\operatorname {U}_2:2V_1 \oplus 2V_1^*$ and $\operatorname{SU}_2:kV_1$ for $k=3,4,5,6$.
Let us remark that $\operatorname {SO}_3:T^*(2\R^3)$ gives rise to the Hilbert series that occurred in Subsection \ref{subsubsec:ExperimentsOn}.

\subsubsection{Non-1-large cases}

We have examined the on-shell Hilbert series for a couple of non-1-large representations. For example we considered the representations $\U_2: k V_1$ for $k=2,3,4,5$. For all of these  representations, the on-shell Hilbert series is palindromic and fulfills $(\operatorname{S}_{\le 60})$, which leads to the speculation  that we might have $I_Z^G=I_J^G$. For $k=2$, the on-shell Hilbert series coincides with the Hilbert series of $\C/\Z_2$. For $k=3$, it coincides with the on-shell Hilbert series of a totally degenerate $\Sp^1:\C^3$ (cf. Subsection \ref{subsec:LaurentCompletDegen}). For the representation $\SU_3: 2 V_1$, the on-shell Hilbert series is non-palindromic and violates $(\operatorname{S}_{r})$ for $r=2,3,4,5,6$ but fulfills $(\operatorname{S}_{r})$ for $r=1$ and $7\le r \le 200$. For the representation $\SU_3: 3 V_1$, the on-shell Hilbert series is non-palindromic and violates $(\operatorname{S}_{r})$ for $1\le r \le 1000$. For the representation $\SU_3: 4 V_1$, the on-shell Hilbert series is non-palindromic and fulfills  $(\operatorname{S}_{r})$ for $r=1,4$ but violates $(\operatorname{S}_{r})$ for $r\in\{1,...,200\}\backslash\{1,4\}$.

\appendix

\section{Schur polynomials}
\label{ap:schur}

Here, we briefly recall the definitions of the Schur polynomials
for the benefit of the reader; see \cite[Section I.3]{MacdonaldBook} and
\cite[Sections 4.4--6]{SaganBook} for more details.

Given $\mu = (\mu_1, \ldots, \mu_n) \in \Z^n$, the \emph{alternant} $a_\mu$ in the variables
$\bs{x} = (x_1, \ldots, x_n)$ is the alternating polynomial
\[
    a_\mu (\bs{x})
    =
    \det( x_i^{\mu_j} )_{1\leq i,j\leq n}.
\]
For $\delta = (n-1, n-2, \ldots, 0)$, the corresponding alternant $a_\delta(\bs{x})$
is called the \emph{Vandermonde determinant} and admits the factorization
\[
    a_\delta(\bs{x})
    =
    \det( x_i^{n-j})_{1\leq i,j \leq n}
    =
    \prod\limits_{1\leq j < k \leq n} x_j - x_k .
\]
In particular, it follows that $a_\delta(\bs{x})$ divides every alternating
polynomial in $\bs{x}$, and hence every alternant $a_\mu(\bs{x})$.

\begin{definition}
\label{def:SchurPoly}
Let $\rho = (\rho_1, \ldots, \rho_n) \in \Z^n$ with $\rho_1 \geq \rho_2 \geq \cdots \geq \rho_n$.
The \emph{Schur polynomial} associated to $\rho$ in the variables $x_1, \ldots, x_n$
is the symmetric polynomial defined by
\[
    s_\rho(\bs{x})
    =
    \frac{a_{\delta+\rho}(\bs{x})}{a_{\delta}(\bs{x})}
    =
    \frac{\det( x_i^{\rho_j + n - j})_{1\leq i,j \leq n}}
        {\det( x_i^{n-j})_{1\leq i,j \leq n}}.
\]
\end{definition}

Alternatively, the Schur polynomials can be characterized as follows.  Recall that a
\emph{generalized tableaux $T$ of shape $\mu = (\mu_1, \ldots, \mu_n) \in \Z^n$} is a left-justified
array with $n$ rows of lengths $\mu_i$ having positive integer entries $T_{i,j}$ with $1\leq i \leq n$
and $1\leq j \leq \mu_i$.  A generalized tableaux is \emph{semistandard} if its rows are nondecreasing
and its columns are increasing.  Given a generalized tableaux $T$ of shape $\mu$, let
\[
    \bs{x}^T
    =
    \prod\limits_{(i,j)\in \mu} x_{T_{i,j}}.
\]
That is, $\bs{x}^T$ is given by the product of all $x_t$ where $t$ ranges over the
entries of $T$.

\begin{lemma}
\label{lem:SchurYoung}
Let $\rho = (\rho_1, \ldots, \rho_n) \in \Z^n$ with $\rho_1 \geq \rho_2 \geq \cdots \geq \rho_n$.
Then
\[
    s_\lambda(\bs{x})   =   \sum\limits_T \bs{x}^T
\]
where the sum is over all semistandard $\lambda$-tableaux $T$ whose entries are elements of $\{ 1, \ldots, n \}$.
\end{lemma}

For the proof, see \cite[Corollary 4.6.2]{SaganBook}, and note that Sagan uses this description
as the definition of the Schur polynomials.  In particular, note that is obvious from this definition
that $s_\lambda(\bs{x}) > 0$ for any $\lambda$ when $x_i > 0$, proving Corollary \ref{cor:LaurentPositive}.

\bibliographystyle{amsplain}
\bibliography{HS}

\end{document}